\documentclass[reqno, 11pt]{amsart}
\usepackage[hmargin=3cm,vmargin=3cm]{geometry}
\usepackage{amsmath,amssymb,amsthm}
\usepackage{verbatim,epsfig,graphics,hyperref}
\usepackage{marginnote}
\usepackage{color}
\usepackage{mathrsfs}
\usepackage{tikz}
\usepackage{svg}
\newtheorem{theorem}{Theorem}[section]
\newtheorem{proposition}[theorem]{Proposition}

\newtheorem{lemma}[theorem]{Lemma}

\theoremstyle{definition}

\newtheorem{convention}[theorem]{Convention}

\newtheorem{example}[theorem]{Example}
\newtheorem{remark}[theorem]{Remark}

\newcommand{\Ck}{\mathcal{C}}

\newcommand{\R}{\mathbb{R}}

\newcommand{\vol}{{\rm vol}}
\newcommand{\Lip}{{\rm Lip}}
\newcommand{\del}{\partial}
\newcommand{\dd}{\, \mathrm{d}}

\newcommand{\g}{\normalfont{\texttt{g}}}
\newcommand{\balpha}{\boldsymbol{\alpha}}
\newcommand{\hball}{{B}_{\scriptscriptstyle{\mathbb{H}^{n-1}}}}
\DeclareMathOperator{\tube}{w}
\DeclareMathOperator{\roll}{roll}
\DeclareMathOperator{\DtN}{\mathscr{D}}
\DeclareMathOperator{\inj}{inj}
\DeclareMathOperator{\dist}{dist}

\newcommand{\ostube}{{\tau}_+}
\newcommand{\thick}{{\scriptscriptstyle{\text{thick}}}}
\newcommand{\thin}{\scriptscriptstyle{\text{thin}}}
\newcommand{\intM}{{M}_{{\thin,\scriptscriptstyle{\circ}}}}
\newcommand{\bM}{{M}_{{\thin,\scriptscriptstyle{\partial}}}}
\newcommand{\double}{\mathcal{D}}
\newcommand{\II}{I\!I}
\newcommand{\newmu}{\eta}

\newcommand{\Vtube}{\mathcal{V}}
\usepackage{enumitem}
\DeclareMathOperator{\supp}{supp}

\DeclareMathOperator{\length}{\textrm{length}}

\newcounter{countc} % Define a new counter for the indices

\newcommand{\countc}{%
  \stepcounter{countc}% Increment the counter for the indices
  \arabic{countc}% Display the counter value
}
\title{Tubes and Steklov eigenvalues  in negatively curved manifolds}
\author[Basmajian]{Ara Basmajian}
\address{The Graduate Center, City University of New York, NY, New York, 10016 and Hunter College,
City University of New York, 695 Park Ave., New York, NY, 10065}
\email{abasmajian@gc.cuny.edu}
\author[Brisson]{Jade Brisson}
\address{Institut de Math\'ematiques, Universit\'e de Neuch\^atel, Rue Emile-Argand 11, 2000 Neuch\^atel, Suisse}
\email{jade.brisson@unine.ch}
\author[Hassannezhad]{Asma Hassannezhad}
\address{University of Bristol,
School of Mathematics,
Fry Building,
Woodland Road,
Bristol, 
BS8 1UG, U.K.}
\email{asma.hassannezhad@bristol.ac.uk}

\author[M\'etras]{Antoine M\'etras}
\address{Institut de Math\'ematiques, Universit\'e de Neuch\^atel, Rue Emile-Argand 11, 2000 Neuch\^atel, Suisse}
\email{antoine.metras@unine.ch}

\begin{document}
\subjclass[2020]{35P15, 58C40}
\maketitle

\begin{abstract}
   
    We consider the Steklov eigenvalue problem on a compact pinched negatively curved manifold $M$ of dimension at least three with totally geodesic boundaries.  We obtain a geometric lower bound for the first nonzero Steklov eigenvalue in terms of the total volume of $M$ and the volume of its boundary. We provide examples illustrating the necessity of these geometric quantities in the lower bound. Our result can be seen as a counterpart of the lower bound for the first nonzero Laplace eigenvalue on closed pinched negatively curved manifolds of dimension at least three proved by Schoen in 1982. 
    
    The proof is composed of certain key elements. We provide a uniform lower bound for the first eigenvalue of the Steklov-Dirichlet problem on a neighborhood of the boundary of $M$ and show that it provides an obstruction to having a small first nonzero Steklov eigenvalue. 
    As another key element of the proof,  we give a tubular neighborhood theorem for totally geodesic hypersurfaces in a pinched negatively curved manifold. We give an explicit dependence for the width function in terms of the volume of the boundary and the pinching constant. 
\end{abstract}

\section{Introduction}
Given a compact connected manifold $M$ of dimension $n$ with boundary $\del M$, we consider the Dirichlet-to-Neumann map $\DtN$:
\begin{eqnarray*}
    \DtN: C^\infty(\partial M) & \to & C^\infty(\partial M) \\
    f & \mapsto & \partial_\nu \tilde f,
\end{eqnarray*}
where $\tilde f$ is the harmonic extension of $f$ to $M$, and $\nu$ is the outward unit normal vector field along $\partial M$. The map $\DtN$ is a first-order elliptic pseudo-differential operator and its spectrum consists of a discrete sequence of non-negative real numbers with the only accumulation point at infinity. We enumerate them in increasing order counting their multiplicities.
\[0 = \sigma_0 < \sigma_1 \le \sigma_2 \le \cdots \nearrow\infty.\]
The eigenvalues of the Dirichlet-to-Neumann map are also the eigenvalues of the \emph{Steklov problem}:
\[\begin{cases}
    \Delta u = 0 & \text{in $M$},\\
    \partial_\nu u = \sigma u &\text{on $\del M$},
\end{cases}\]
which was considered by Steklov \cite{Stek} in 1902. Hence, they are usually referred to as the Steklov eigenvalues. 

We consider Steklov eigenvalues on a compact connected manifold $M$ of dimension $n$, $n\ge3$,  with sectional curvature in the interval $[-1, -\kappa^2]$ for some $\kappa \in (0,1]$, and having totally geodesic boundary. We call such manifold \emph{$(-\kappa^2)$–\,pinched negatively curved $n$–\,manifolds} with totally geodesic boundary. The main focus will be on geometric lower bounds for the first nonzero Steklov eigenvalue in this setting.  

There has been significant research concerning the connection between Steklov eigenvalues and the geometric properties of the underlying space. We refer to the survey papers~\cite{CGGS, GP17} for the developments on the subject over the last two decades. 
However, not much has been explored in the context of pinched negatively curved $n$–\,manifolds, $n\ge 3$, with totally geodesic boundary. 
Given the numerous studies on geometric bounds for Laplace eigenvalues on pinched negatively curved manifolds, see e.g. \cite{HV22, Ham19, BCD93, Dod, DR86, Schoen82, SWY80}, and a close connection between Laplace and Steklov eigenvalues, the aforementioned setting provides a natural context for exploring geometric bounds for Steklov eigenvalues and investigating how the geometric and topological properties of $M$ and $\del M$ influence eigenvalue bounds. 
The Steklov eigenvalues are sensitive to the geometry near the boundary of the manifold {(see e.g. \cite{Per23,CGH20,CGGS})}. As such, we are led to study the geometry near totally geodesic hypersurfaces.

There is a rich literature on this subject, in particular in the hyperbolic setting. In dimension two, the well known \emph{collar lemma} on hyperbolic surfaces, dating back to Keen \cite{Ke74}, states that short (or long) closed geodesics have a large (or small) tubular neighborhood. The main point being that the width of the tube depends only on the local geometry about the closed geodesic.  Buser \cite{bus78, Bus92}, Chavel and Feldman \cite{CF78}, Basmajian \cite{Bas92, Bas93a,Bas94}, and Kojima \cite{Koj96} have proven various versions of the collar lemma and extended it to the setting of negatively curved surfaces. 
In dimension at least 3, the collar lemma generalize to the so-called \emph{tubular neighborhood theorem} for embedded closed totally geodesic hypersurfaces of a hyperbolic manifold \cite{Bas92, Bas94}. It states that a totally geodesic hypersurface has a tubular neighborhood whose width only depends on the $(n-1)$–\,volume of the hypersurface. A related result with the width depending on the $n$–\,volume of the manifold was derived by Zeghib \cite{Zeg91}. Note that in the hyperbolic setting the $(n-1)$–\,volume of the hypersurface is bounded above by the $n$–\,volume of the manifold \cite{Zeg91,miy94}. The tubular neighborhood theorem has many applications and consequences, see for example \cite{KM, Zeg91, Bas93a, Bas94, BB22, PR22}. 

 We extend the tubular neighborhood theorem to the setting of pinched negatively curved $n$–\,manifolds, $n\ge 3$, which will be of independent interest. We shall see that the tubular neighborhood theorem plays a crucial role in obtaining geometric bounds for the Steklov eigenvalues. 
To state the theorem, we introduce the following functions.

Let $r(x) := \log\coth \frac{x}{2}$ and let $V_{n,\kappa}(r)$ be the $n$–\,volume of a hyperbolic ball of radius $r$ in the space form $\mathbb{H}^{n}(-\kappa^2)$ of constant sectional curvature $-\kappa^2$. The \emph{$(n, \kappa)$–\,width function} (or simply the width function) $\tube_{n,\kappa}(A)$ is defined to be
\begin{equation}\label{eq:w_n,k}
    \tube_{n,\kappa}(A):=\frac{1}{2}(V_{n-1,\kappa}\circ r)^{-1}(A), \qquad A\in(0,\infty).
\end{equation}
When $\kappa=1$, we simply denote the width function by $\tube_{n}(A)$. 

We now state the tubular neighborhood theorem extending the result in \cite{Bas94} to the setting of pinched negatively curved $n$–\,manifolds.
\begin{theorem}[Tubular neighborhood theorem]\label{thm:tube} 
Let $N$ be a complete, {orientable,} $(-\kappa^2)$–\,pinched negatively curved $n$–\,manifold, containing {a closed} embedded totally geodesic hypersurface $\Sigma$. Let $A$ be the volume of $\Sigma$. Then $\Sigma$ has a tubular neighborhood 
\[\{x\in N : \dist(x,\Sigma)<\tube_{n,\kappa}(A)\}\]
of width $\tube_{n,\kappa}(A)$ which is diffeomorphic to $(-\tube_{n,\kappa}(A),\tube_{n,\kappa}(A))\times \Sigma$. 
Moreover, the tubular neighborhoods of disjoint closed totally geodesic hypersurfaces are mutually disjoint. 
\end{theorem}

We use the comparison geometry and properties of the normal exponential map to extend the spirit of the proof in \cite{Bas94} to the negatively curved manifolds. 

 In dimension two, the key role of the collar lemma in the study of geometric bounds for both Laplace and Steklov eigenvalues is apparent, see \cite{SWY80, Bus80, burger88, Per22, HMP}. We show that the tubular neighborhood theorem in pinched negatively curved manifolds plays an important role in obtaining lower bounds for the Steklov eigenvalues {in higher dimensions}. 
 
 Let $\{\lambda_k(\del M)\}$ be the sequence of Laplace eigenvalues of the boundary $\del M$. A standard result in spectral theory implies that (see e.g. \cite[Chapter 15]{Shu})
$$\sigma_k(M) \sim \sqrt{\lambda_k(\partial M)}, \qquad \text{as $k \to \infty$}.$$
Several works have investigated quantitative improvement to this result, and, in particular, comparison of $\sigma_k(M)$ and $\sqrt{\lambda_k(\partial M)}$ holding for all $k \geq 1$. This study was initiated by Provenzano and Stubbe \cite{PS19} in the Euclidean setting and has been extended to the Riemannian setting in \cite{Xio18, CGH20, GKLP22}. Colbois, Girouard and Hassannezhad \cite{CGH20} proved a uniform bound
\begin{equation}\label{inq:cgh}
|\sigma_k(M) - \sqrt{\lambda_k(\del M)}|\le C, 
\end{equation}
where $C$ is an explicit geometric quantity independent of $k$. When $M$ is a $(-\kappa^2)$–\,pinched negatively curved $n$–\,manifold with totally geodesic boundary, the constant $C$ can be expressed in terms of the dimension and the \emph{rolling radius} (the distance between $\partial M$ and its cut locus).  
 The following theorem is a consequence of tubular neighborhood theorem~\ref{thm:tube} and inequality \eqref{inq:cgh} together with Schoen's result \cite{Schoen82} on the lower bound for the first nonzero Laplace eigenvalues of negatively curved $n$–\,manifolds, $n\ge3$.
\begin{theorem}\label{thm:Schoencounterpart}
    Let $M$ be a compact $(-\kappa^2)$–\,pinched negatively curved $n$–\,manifold, $n \geq 4$ with totally geodesic boundary. Let $b$ be the number of boundary components of $M$ and $A$ the maximum $(n-1)$–\,volume of the boundary components. Then 
    \begin{equation*}
        \sigma_b(M) \geq \frac{C_1(n, \kappa)}{A^{2 + \frac{1}{\kappa(n-2)}}},
    \end{equation*}
    where $C_1(n, \kappa)$ is a positive constant depending only on $n$ and $\kappa$.
\end{theorem}

Note that inequality \eqref{inq:cgh}  does not lead to any non-trivial lower bounds for $1\le k\le b-1$. We can construct examples (see Example \ref{ex:smalleigenvalue}) showing that the lower bound given in Theorem~\ref{thm:Schoencounterpart} fails to hold for $1\le k\le b-1$.  We also show that there exists a sequence of  hyperbolic $n$–\,manifold $M_\epsilon$, $n\ge3$, with  totally geodesic boundary with two connected components such that the volume of the boundary goes to infinity and for every $k\ge1$
    \begin{align*}
       \lim_{\epsilon\to0} \sigma_k(M_\epsilon) =0, 
    \end{align*}
This shows the necessity of the presence of $A$ in the lower bound.  
    In dimension 3, such a sequence of hyperbolic 3-manifolds can be constructed with a connected totally geodesic boundary, see Remark \ref{rmk:small-eigenvalues}.  
We do not believe that the power of $A$ in Theorem \ref{thm:Schoencounterpart} is optimal. It is interesting to identify the optimal power.

If we use the same approach together with the result of  Schoen, Wolpert, and Yau \cite{SWY80} for closed negatively curved surfaces, we can obtain a lower bound for $\sigma_b(M)$ for pinched negatively curved 3-manifolds, see Section \ref{sec:stekbound} for details.  In this case, the lower bound goes to zero when there is a shrinking simple closed geodesic on the boundary that cuts the surface boundary into two connected components. In particular, 
   let $M$ be a $(-\kappa^2)$–\,pinched negatively curved $3-$manifold with a connected totally geodesic boundary of genus $g$.   Then there exists a positive constant $C_2(\kappa,g)$  depending only on $\kappa$ and the genus $g$ of $\partial M$, such that 
   \[\sigma_1(M)\ge {C_2(\kappa,g)\ell_1(\partial M)}\]
where $\ell_1(\partial M)$ denotes the length of a curve consisting of a disjoint union of simple closed geodesics that divides $\partial M$ into two connected components and has the minimum length.

By Schoen, Wolpert, and Yau's result \cite{SWY80}, see also Buser's inequality \cite{Bus80}, we know that  $\lambda_1(\partial M)$ goes to zero when there is a shrinking simple closed geodesic cutting the surface into two connected components. However, it is not known if one can get an upper bound in terms of $\ell_1(\partial M)$ for $\sigma_1(M)$.

When the boundary is not connected, we do not get any nontrivial lower bound for $\sigma_1(M)$ following the approach discussed above and we need to use a completely different approach. The following theorem
can be viewed as a counterpart of Schoen's lower bound \cite{Schoen82} for the first nonzero Laplace eigenvalue of a closed pinched negatively curved $n$–\,manifold, $n\ge3$. 
 
 \begin{theorem}\label{thm:specgap-intro}
Let $M$ be a compact $(-\kappa^2)$–\,pinched negatively curved $n$–\,manifold, $n \geq 3$, with totally geodesic boundary. Let $b$ be the number of boundary components of $M$, $A$ the maximum $(n-1)$–\,volume of the boundary components, and $V$ the $n$–\,volume of $M$. Then 
   \[\sigma_1(M)\ge \frac{C_{3}(n,\kappa)}{bA^{\frac{2n}{\kappa(n-2)}}V^2},\]
   where $C_3(n,\kappa)$ is a positive  constant  depending only on $n$ and $\kappa$.

\end{theorem}
In section \ref{subset:sigma1}, we give an improved version of Theorem \ref{thm:specgap-intro} where the exponent of $V$ is one. Moreover, Example \ref{ex:smalleigenvalue} shows that this exponent is optimal.
 Note that $b$ and $A$ are both bounded in terms of $V$ as a result of the following inequality \cite[\S 4.4, \S 4.7]{Zeg91}
 \begin{equation}\label{inq:az}
     \vol(M)\ge c(n,\kappa)\vol(\partial M),
 \end{equation}
 where $c(n,\kappa)$ is a positive constant. 
 
 Thus,
when $b=1$ and $n\ge4$, Theorem \ref{thm:Schoencounterpart} gives a better lower bound, in particular for large $A$. When  $b=1$ and $n=3$,  we get the following lower bound.  
\begin{eqnarray*}
    \sigma_1(M)\ge\max\left\{C_{2}(g,\kappa)\ell_1(\partial M),\frac{C_{4}(g,\kappa)}{V^2}\right\}.
\end{eqnarray*}

It is an interesting question to see whether $V$ has to go to infinity when $\ell_1(\partial M)$ tends to zero. It relates to the volume of thin tubes. There are lower bounds for the volume of the thin tubes in pinched negatively curved manifolds \cite{Bus80,BCD93}. However, in dimension 3, the lower bound does not go to infinity as the length of the base geodesic goes to zero. 

According to the result of Fujii and Soma \cite{fujii1997totally}, the set of all hyperbolic structures on a surface of genus at least two, such that the surface can be realized as the totally geodesic boundary of a compact hyperbolic 3-manifold, is dense in the moduli space. Consequently, we can have a family of hyperbolic 3-manifolds with one boundary component with a pinching geodesic on its boundary. Hence, the above question is also interesting in the hyperbolic setting, and to our knowledge, it is not known, even in the hyperbolic context. 
It also relates to systolic inequalities in the hyperbolic setting. There are versions of the systolic inequalities for hyperbolic manifolds, see for example \cite{BT11,BB22}. However, they do not provide an answer to the above question.

The proof of Theorem \ref{thm:specgap-intro} is inspired by the work of Dodziuk and Randol \cite{DR86}, where they provide a new proof for Schoen's and Schoen-Wolpert-Yau's inequalities. We adapt a similar strategy; however, the presence of the boundary poses some technical challenges. We consider an adapted version of the thick-thin decomposition for manifolds with totally geodesic boundary using the tubular neighborhood theorem.  We then show that the oscillation of $\sigma_1$–\,eigenfunction on the thick part of the manifold away from the boundary is bounded above in terms of  $\sigma_1(M)$ taking into account its variational characterization: 
\begin{equation}\label{eq:variational characterisation}
\sigma_1(M)=\inf\left\{\frac{\int_{M}\,|\nabla f|^2 \dd v}{ \int_{\partial M}f^2\dd s}: f\in H^1(M),\int_{\partial M}f=0\right\}.\end{equation}
 We establish a uniform lower bound for the first Steklov-Dirichlet eigenvalue on a tubular neighborhood of the boundary union with a subfamily of the thin part of the manifold and show that it plays an obstruction for having a small $\sigma_1(M)$.
 
 For further application of this approach on obtaining lower bounds for the Steklov eigenvalues in dimension 2 we refer to \cite{HMP}, see also Remark \ref{rem:HMP}. 

The original proof by Schoen \cite{Schoen82} for establishing a lower bound for the first Laplace eigenvalue of a negatively curved $n$–\,manifold, where $n \ge 3$, uses the Cheeger inequality and a bound on the Cheeger constant. Additionally, Buser obtained an upper bound in terms of the Cheeger constant for the first nonzero Laplace eigenvalue \cite{Bus82}. Hence, the Cheeger constant reflects the behavior of the first nonzero Laplace eigenvalue.

Cheeger-type inequalities for the Steklov eigenvalues have been explored in \cite{esc97,esc99,Jam15,HM20}.
In dimension 2, Perrin \cite{Per22}, using a modification of the Cheeger-Jammes constant, obtained a counterpart of Schoen-Wolpert-Yau's inequalities. However, estimating isoperimetric constants that appear in Cheeger-type inequalities for the Steklov eigenvalues seems to be very challenging in higher dimensions. Moreover, there is no guarantee that they provide an `optimal' lower bound, as there is no version of a Buser-type inequality in terms of those isoperimetric constants for the first nonzero Steklov eigenvalue. 
\\

    The paper is organized as follows. We give a brief overview of the geometric properties of negatively curved manifolds and a definition of the thick-thin decomposition which will be used in this paper in Section \ref{sect: preliminaries}. We then prove the tubular neighborhood theorem in Section \ref{sect:tube}. In Section \ref{sec:stekbound}, we discuss the geometric bounds for the Steklov eigenvalues and prove Theorems \ref{thm:Schoencounterpart} and \ref{thm:specgap-intro}. Section \ref{sec:smallstek} deals with constructing examples of negatively curved manifolds with small Steklov eigenvalues showing that the geometric quantities appearing in the main theorems are necessary.

    \subsection*{Acknowledgement}The authors would like to thank Bruno Colbois and Jozef Dodziuk for valuable comments, and Abdelghani Zeghib for the insightful explanation of his result in  \cite[\S4.4 Proposition~3]{Zeg91}. The authors would also like to thank the anonymous referee for helpful comments and suggestions, which led to an improvement in the presentation of the paper.  A.\,B. is supported by PSC CUNY Award 65245-00 53. J.\,B. acknowledges support of the SNSF project ‘Geometric Spectral Theory’, grant number 200021-19689. A.\,H. and A.\,M. acknowledge  support of EPSRC grant EP/T030577/1.

\section{Preliminaries and Background}\label{sect: preliminaries}

In this section, we will review some fundamental and well-known facts about the geometry of smooth compact negatively curved manifolds, which will be used in the subsequent sections. It is important to note that throughout the paper we assume that manifolds are both compact and connected unless stated otherwise. Thus, we may not reiterate these assumptions.

A $(-\kappa^2)$–\,\textit{pinched negatively curved manifold} is an $n$–\,manifold with sectional curvature  in the interval $[-1,-\kappa^2]$ for some positive constant $\kappa\in(0,1]$.  

For the study of the Steklov eigenvalue problem in Sections \ref{sec:stekbound} and \ref{sec:smallstek}, we, in particular, consider $(-\kappa^2)$–\,pinched negatively curved manifold $M$ with \textit{totally geodesic boundary}, meaning that every geodesic on the boundary with its induced Riemannian metric is also a geodesic on the manifold. Such manifolds admit a double $\double M$ which is a closed $(-\kappa^2)$–\,pinched negatively curved manifold. 

For a closed $(-\kappa^2)$–\,pinched negatively curved manifold $N$, Margulis' Lemma, see e.g. \cite[\textsection 10]{BGS85}, implies that there exists a positive constant $\mu=\mu(n)\le 1$ depending only on the dimension\footnote{Note that in general, the Margulis constant $\mu$ for negatively curved manifolds depends only on the dimension and the lower bound of the sectional curvature which  is $-1$ in our case.} such that $N$ can be viewed as the disjoint union of two sets, the thick part 
$$\tilde{N}_{\thick}=\{x\in N \,:\, \inj_{N}(x)\ge {\mu}\}$$
which is always nonempty, and the thin part 
$$\tilde N_{\thin}{=\{x\in N: \inj_N(x)<\mu\}}$$ which is either empty or consists of a disjoint union of finitely many \textit{thin} tubes $\tilde T_\gamma$ with the following geometry. 
Every thin tube $\tilde T_\gamma$ 
 is a tubular neighborhood of a simple closed geodesic $\gamma$ of length strictly less than $2\mu$ and is homeomorphic to $\gamma\times B^{n-1}$, where $B^{n-1}$ is an $(n-1)$–\,dimensional ball in $\R^{n-1}$. There exists a unique simple closed geodesic of length strictly less than $2\mu$ in each thin tube.
 For any $x\in \gamma$ and $v$ a unit vector orthogonal to $\gamma$ at $x$, let $R=R_{(x,v)}$ represent the maximal length of the unit speed geodesic ray emanating from $\gamma$ at the point $x$ in the direction of $v$ such that the injectivity radius of points on the ray is strictly less than $\mu$. Note that in dimension three, $R$ is the same for all $x\in \gamma$ and $v$.  Tube $\tilde T_\gamma$ is the union of the image of all such geodesic rays with maximal length.  Moreover, when the length of $\gamma$ goes to zero,  $R$ tends to $\infty$. We refer to \cite{Bus80,BCD93}  for more details.
 
 By definition, for every $x$ on the boundary of a thin tube $\tilde T_\gamma$ we have $\inj_N(x)=\mu$. The boundary of $\tilde T_\gamma$ is not necessarily smooth when $n\ge4$. We want to modify the definition of the thin part so that it has a smooth boundary. It will be useful when we apply the Green formula on the thin part, and  for some steps in the proof of Theorem \ref{thm:specgap-intro}.
 
  Colbois, Buser, and Dodziuk \cite[Theorem 2.14]{BCD93} showed that there exists a positive constant $\newmu=\newmu(n)<\frac{\mu}{2}$ depending only on  $n$ such that
 for every $\gamma$ with $\length(\gamma)< \newmu$, and  $ T_\gamma^*=\{x\in\tilde T_\gamma: \inj_N(x)<\frac{\mu}{2}\}$ the following holds.
 \begin{itemize}
 \item[-]$\dist(\partial T_\gamma^*,\gamma)\ge 10$;
 \item[-] there exists a smooth hypersurface $H_\gamma\subset \tilde T_\gamma\setminus \gamma$ which is homeomorphic to $\partial T_\gamma^*$. The homeomorphism is given by pushing along the radial direction;
 \item[-]  the distance between $x^*\in \partial T_\gamma^*$ and $x\in H_\gamma$, both lying on the same ray orthogonal to $\gamma$, is less than $\mu/50$. In particular $\dist(\partial T_\gamma^*,H_\gamma)\le \mu/50$. 
 \end{itemize}
 For any $\gamma$ with $\length(\gamma)< \eta$, let $T_\gamma$ denote the connected component of $\tilde T_\gamma \setminus H_\gamma$ that contains $\gamma$. Note that $\dist(\partial T_\gamma, \gamma)\ge 9$. 
 
 We now give a slightly different definition of the thick and thin parts and use this new definition when we refer to the thick-thin decomposition throughout the paper.
We redefine the thick part, $N_{\thick}$, to be
 \[N_{\thick}= \left\{x\in N \,:\, \inj_{N}(x)\ge \frac{\newmu}{2}\right\},\]
 and thin part $N_{\thin}$ to be the union of all thin tubes $T_\gamma$ defined above
 \[N_{\thin}=\hspace{-5mm}\bigcup_{\substack{\gamma \\ \length(\gamma)<\eta}}\hspace{-5mm}T_\gamma.\]
 We also define the \textit{shell} of $T_{\gamma}$ to be 
$$S_\gamma=\left\{x\in T_\gamma\,:\, \dist(x,\partial T_{\gamma})=\dist(x,H_{\gamma})\le \frac{\mu}{5}\right\}.$$ 

Using the inequality $\inj_N(y)\ge \inj_N(x)-\dist(x,y)$ (see e.g. \cite{BCD93,Xu}), it follows from the definition that for every $x\in \partial T_\gamma=H_\gamma$, we have $\inj_N(x)\ge \frac{\mu}{2}-\frac{\mu}{50}\ge \frac{\newmu}{2}$. In addition, we have $\inj_N(y)\ge\frac{24\mu}{50}\ge\frac{\newmu}{2}$ for any $y\in S_\gamma$.
It implies that 
\[S_\gamma\subseteq N_{\thick}\cap N_{\thin}.\]

Margulis' Lemma also implies that the thick part is connected when $n\ge3$.  This gives a lower bound on the volume of a compact $(-\kappa^2)$–\,pinched negatively curved manifold. Using the G\"unther–Bishop volume comparison theorem \cite[Theorem III.4.2]{chavel_2006},  it can be easily seen that the volume of a $(-\kappa^2)$–\,pinched negatively curved manifold $N$ is bounded below by the volume of a ball of radius $\mu$ in $\mathbb{H}^{n}(-\kappa^2)$.
\begin{equation}\label{inq:lowerv}
    \vol(N)\ge V_{n,\kappa}(\mu)\ge \omega_n\mu^n,
\end{equation}
where $\omega_n$ is the volume of the unit Euclidean ball.

For a $(-\kappa^2)$–\,pinched negatively curved manifold $M$ with {totally geodesic boundary}, its thick-thin decomposition is the projection of the thick-thin decomposition of $\double M$ under the canonical map $\double M\to M$. Equivalently, we can view $M$ as a subdomain of $\double M$ and define $$M_{\thin}=(\double M)_{\thin}\cap M,\qquad M_{\thick}=(\double M)_{\thick}\cap M.$$ 

Similarly, $M_{\thick}$ is nonempty and is connected when $n\ge3$. Each connected component of $M_{\thin}$ is {the intersection of} a tube {in $(\double M)_{\thin}$ with $M$} based on a unique closed simple geodesic $\gamma$ with $\length(\gamma)<\newmu$. Hence, each connected component of $M_{\thin}$ is diffeomorphic to one of the following sets:
\begin{itemize}
    \item[] Type I.~ $\gamma\times B^{n-1}$ where $B^{n-1}$ is an $(n-1)-$ ball in $\R^{n-1}$;
    \item[] Type II.~$\gamma\times B^{n-1}_+$ where $B^{n-1}_+$ is a half ball; 
    \item[] Type III.~ $\alpha \times B^{n-1}$, where $\alpha$ is a geodesic arc and its double $\double \alpha$ is equal to $\gamma$.
\end{itemize}
By a slight notational abuse, we refer to each component of $M_{\thin}$ as a tube and denote it by $T_\gamma$ or $T_\alpha$.
We denote the collection of Type I tubes by $\intM$, and the collection of Type II and Type III tubes by $\bM$. Thus, we can write $$M_{\thin}=\intM\sqcup \bM.$$

Later, we consider $\epsilon$–\,\textit{thick-thin decomposition}, $\epsilon\le\mu$, where $\epsilon$ may depend on some geometric invariants. For any $\epsilon\le \mu$,  $\epsilon$–\,\textit{thick part}, $M_{\thick}^\epsilon$, and $\epsilon$–\,\textit{thin part}, $M_{\thin}^\epsilon$, are defined as follows and have the same properties discussed above.
$$M_{\thick}^\epsilon=\left\{x\in \double M:\inj_{\double M}(x)\ge \min\{\epsilon, \newmu/2\}\right\}\cap M$$
and $M_{\thin}^\epsilon$ is the union of all type I and II tubes $T_{\gamma}$ in $\double M_{\thin}{\cap M}$  with $\length(\gamma)< \min\{2\epsilon,\newmu\}$ together with type III tubes $T_{\alpha}$ with $2\length(\alpha)< \min\{2\epsilon,\newmu\}$.\\

We end this section with the following standing convention which will be used throughout the paper. \\

\begin{convention}\label{convention} Throughout the paper, the notation $c_i(\cdot)$ denotes a positive constant that depends only on certain specified arguments, whereas the notation $c_i$ represents a positive universal constant. It is important to note that we may reset the enumeration of the index in each section, so in a new section, the constant $c_i(\cdot)$ or $c_i$ might refer to different values.
\end{convention}

\section{Tubular neighborhood theorem}\label{sect:tube}
Let us recall the definition of the $(n,\kappa)$–\,width function $\tube_{n,\kappa}$ given in  \eqref{eq:w_n,k}.
\[ \tube_{n,\kappa}(A):=\frac{1}{2}(V_{n-1,\kappa}\circ r)^{-1}(A), \qquad A\in(0,\infty).\]
When $\kappa=1$, we drop $\kappa$, i.e.  $\tube_{n}(x)=\tube_{n,1}(x)$ and $V_n(r)=V_{n,1}(r)$. The width function $\tube_{n,\kappa}(A)$ is monotone decreasing in $A$ and tends to zero as $A\to\infty$, and is monotone increasing in $\kappa$.

The main goal in this section is to extend the tubular neighborhood theorem in the hyperbolic setting \cite{Bas94} to a tubular neighborhood theorem for embedded totally geodesic hypersurfaces in a complete $(-\kappa^2)$–\,pinched manifold with $\kappa\in(0,1]$ for which the width is given by $\tube_{n,\kappa}$ as a function of the volume of the hypersurface.

\begin{theorem}[Tubular neighborhood theorem]\label{thm:tubeII} 
Let $N$ be a complete, orientable,  $(-\kappa^2)$–\,pinched $n$–\,manifold containing an embedded closed totally geodesic hypersurface $\Sigma$. Let $A$ be the volume of $\Sigma$. Then $\Sigma$ has a tubular neighborhood 
\[\{x\in N \,:\, \dist(x,\Sigma)<\tube_{n,\kappa}(A)\}\]
of width $\tube_{n,\kappa}(A)$ which is diffeomorphic to $(-\tube_{n,\kappa}(A),\tube_{n,\kappa}(A))\times \Sigma$ via 
\begin{eqnarray*}\exp^\perp:(-\tube_{n,\kappa}(A),\tube_{n,\kappa}(A))\times \Sigma \to N\\
(t,p)\mapsto \exp(t\nu(p)),
\end{eqnarray*}
where $\nu$ is the unit outward normal vector field along $\Sigma$.\\
Moreover, the tubular neighborhoods of disjoint closed embedded hypersurfaces are mutually disjoint. 
\end{theorem}

Note that when $\kappa=1$, the  tubular neighborhood is isometric to the warped product space $(-\tube_n(A),\tube_n(A))\times_{\cosh}\Sigma$, see \cite{Bas94}.
The key ingredient of the proof of the tubular neighborhood theorem \ref{thm:tubeII} is the following proposition.

\begin{proposition}\label{prop:r(a)}
    Let $N$ be a complete, orientable,  $(-\kappa^2)$–\,pinched $n$–\,manifold containing disjoint  embedded closed totally geodesic hypersurfaces $\Sigma_{1}$ and 
    $\Sigma_{2}$ ($\Sigma_{1}$ may equal $\Sigma_{2}$). Let $\alpha$ be a geodesic arc from $\Sigma_{1}$ to 
    $\Sigma_{2}$ which is orthogonal to $\Sigma_{1}$ and $\Sigma_{2}$ at its endpoints, and let $a=\length(\alpha)$. Then $\Sigma_{i}$, for $i=1,2$, contains an embedded ball of radius
    \begin{align*}
        r(a) = \log \coth \frac{a}{2}.
    \end{align*}
\end{proposition}
To prove this proposition, we need a triangle comparison theorem for ideal triangles with a right angle. We call a triangle an ideal triangle if it has at least one vertex at infinity. We denote the vertex at infinity by $\infty$.  It is a consequence of the classic Alexandrov-Toponogov triangle comparison theorem \cite[Theorem IX.5.2]{chavel_2006}.

\begin{lemma}[Ideal triangle comparison lemma] \label{lem:triangle-comp}
    Let $N$ be a complete simply connected $n$–\,manifold with sectional curvature bounded below by $\balpha$, where $\balpha\in(-\infty,0)$.
    Given an ideal triangle $T$ in $N$ with vertices $p, q\in N$ and $\infty$ a vertex at infinity, and with a right angle at vertex $q$ ($\underset{T}{\measuredangle}  q=\frac{\pi}{2}$), there exists an ideal triangle $\bar{T}$ in the hyperbolic plane $\mathbb{H}^2(\balpha)$ of constant curvature $\balpha$ with vertices $\bar{p}, \bar{q}\in \mathbb{H}^2(\balpha)$, and $\infty$, and with right angle at $\bar q$ ($\underset{\bar T}{\measuredangle} \bar q=\frac{\pi}{2}$), such that $$\dist(\bar{p}, \bar{q})=\dist(p,q),\quad\text{and}\quad  \underset{\bar T}{\measuredangle}\bar p\le \underset{ T}{\measuredangle}  p.$$ 
    Note that the ideal triangle $\bar T$ is uniquely determined up to isometry. 
\end{lemma}
\begin{proof}
    For $d > 0$, consider the triangle $T_d$ with vertices $p, q$ and $s_d$, where $s_d$ is a point on the side $[q,\infty]$ of $T$ at a distance $d$ of {$q$}.
    Then by the Alexandrov-Toponogov triangle comparison theorem \cite[Theorem IX.5.2]{chavel_2006}, there exists a triangle $\bar{T}_d$ in the hyperbolic plane $\mathbb{H}^2(\balpha)$ with vertices $\bar{p}_\circ, \bar{q}_d,$ and $ \bar{s}_d$ such that
    $$\dist(\bar{p}_\circ, \bar{q}_d) = \dist(p, q),\quad \dist(\bar{p}_\circ, \bar{s}_d) = \dist(p, s_d),\quad \dist(\bar q_d,\bar s_d)=\dist(q,s_d)=d;$$
    $$\underset{\bar T_d}{\measuredangle} \bar p_\circ\le \underset{ T}{\measuredangle} p,\qquad \underset{\bar T_d}{\measuredangle} \bar q_d\le \underset{ T}{\measuredangle} q=\frac{\pi}{2},\qquad \underset{\bar T_d}{\measuredangle} \bar s_d\le \underset{ T}{\measuredangle} s_d.$$
    \begin{figure}[ht]
    \centering
    
        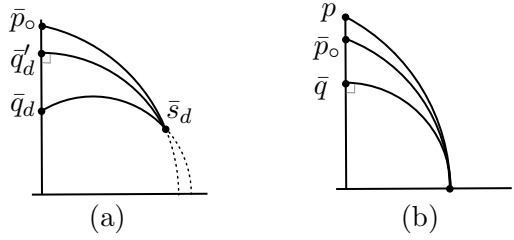
   \vspace{3mm}
   
    \caption{ 
    Picture (a) illustrates the triangles $\bar T_d$ and $\bar T'_d$, and picture (b) the triangles  $\bar T_
    \infty$ and $\bar T$. }
    \label{fig:3}
\end{figure}
    
    Since $\underset{\bar T_d}{\measuredangle}\bar{q}_d \leq \frac{\pi}{2}$ and $\underset{\bar T_d}{\measuredangle} \bar p_\circ<\frac{\pi}{2}$, we can find a point $\bar{q}'_d$ on the side $[\bar{p}_\circ, \bar{q}]$ of $\bar{T}_d$ such that the triangle $\bar{T}'_d$ with vertices $\bar{p}_\circ$, $\bar{q}'_d$, and $\bar{s}_d$ has a right angle at $\bar q'_d$ ($\underset{\bar T'_d}{\measuredangle}\,\bar{q}'_d=\frac{\pi}{2}$), see Figure~\ref{fig:3}\,(a). Then by construction,
    \begin{align*}
        \dist(\bar{p}_\circ, \bar{q}'_d) \leq \dist(\bar{p}_\circ, \bar{q}_d) \leq \dist(p, q),\qquad \underset{\bar T'_d}{\measuredangle} \bar p_\circ=\underset{\bar T}{\measuredangle} \bar p_\circ\le \underset{ T}{\measuredangle} p.
    \end{align*}
    We now show that when $d \to \infty$,  $\bar{T}'_d$ converges to an ideal triangle $\bar{T}_\infty$ with vertices $\bar p_\circ$, $\bar q=\displaystyle{\lim_{d\to\infty} \bar q'_d}$, and $\infty=\displaystyle{\lim_{d\to\infty} }\bar s_d$.\\
    Since $\dist(\bar{p}_\circ, \bar{s}_d)=\dist({p}, {s}_d)$ goes to $\infty$ as $d\to\infty$, the triangle $\bar{T}'_d$ should go to an ideal triangle as $d \to \infty$. The only problem that can arise is $\bar{q}'_d$ converging to $\bar{p}_\circ$ as $d \to \infty$, making the triangle collapse. But as  $\underset{\bar T_d}{\measuredangle}\bar{q}'_d=\frac{\pi}{2}$ for any $d$, and  $\underset{\bar T_d}{\measuredangle}\bar{p}<\frac{\pi}{2}$, we conclude that $\bar q\ne \bar p_\circ$ and $\bar T_\infty$ is an ideal triangle with \begin{align*}
        \dist(\bar{p}_\circ, \bar{q}) \leq \dist(p, q),\qquad \underset{\bar T_\infty}{\measuredangle} \bar p_\circ\le \underset{ T}{\measuredangle} p,\qquad \underset{\bar T_\infty}{\measuredangle} \bar q=\frac{\pi}{2}.
    \end{align*} 
  If $\dist(\bar{p}_\circ, \bar{q}) = \dist(p, q)$, then $\bar T_\infty$ is the desired ideal triangle. Otherwise,  let $\bar p\in \mathbb{H}^2(\balpha)$ be such that the geodesic segment  $[\bar q, \bar p_\circ]$ is a subset of the geodesic segment $[\bar q,\bar p]$ and $\dist(\bar p,\bar q)=\dist(p,q)$. Consider the ideal triangle  $\bar T$ with vertices $\bar p$, $\bar q$, and $\infty$. As illustrated in Figure~\ref{fig:3}\,(b), we have $\underset{\bar T}{\measuredangle} \bar p<\underset{\bar T_\infty}{\measuredangle} \bar p_\circ$. Hence, $\bar T$ is the desired ideal triangle. It is clear that it is unique up to isometry. 
\end{proof}

\begin{proof}[Proof of Proposition \ref{prop:r(a)}]
    Let $\tilde{N}$ be the universal cover of $N$ and let $\tilde{\Sigma}_{1}$ and $\tilde{\Sigma}_{2}$ 
    be connected components of lifts of ${\Sigma}_{1}$ and ${\Sigma}_{2}$ connected by a lift $\tilde{\alpha}$ of $\alpha$. Since $N$ has negative curvature and $\Sigma_i$, $i=1,2$  are totally geodesic, the universal covering map from 
    $\tilde{N}$ to $N$ restricts  to 
    be the universal covering map from  $\tilde{\Sigma}_i$  
    to $\Sigma_{i}$. Now, $\alpha$ determines a
non-trivial relative free homotopy class of arc  in $N$ from $\Sigma_1$ to $\Sigma_2$ . Here relative means the end  points of the arc are allowed to
move in $\Sigma_1$ and $\Sigma_2$, resp. This translates to the fact that the two lifts $\tilde{\Sigma}_1$ and $\tilde{\Sigma}_2$  are disjoint, for otherwise 
$\tilde{\Sigma}_1=\tilde{\Sigma}_2$ and 
there would be a geodesic in $\tilde{N}$ that starts and ends perpendicular to  $\tilde{\Sigma}_1$ violating the fact that there are no geodesic bigons in a simply connected negatively curved space.
    
    Let  $\pi: \tilde{N} \to \tilde{\Sigma}_1$ be the projection map that maps $x\in \tilde N$ to  $\pi(x)\in \tilde \Sigma_1$  the base of the unique geodesic orthogonal to $\tilde{\Sigma}_1$ and passing through {$x$}. The map $\pi$ is {distance non-increasing}. {The domain of $\pi$ can be extended to the boundary at infinity of $\tilde N$; any point on the boundary at infinity $\partial_\infty\tilde N$ of $\tilde N$ can be uniquely mapped to a point in $\tilde \Sigma$.}  Moreover $\pi(\tilde{\Sigma}_{2})$ is diffeomorphic to $\tilde{\Sigma}_2$, see e.g. \cite{BO69} and \cite[Proposition 2.4]{BH99} for more details.
    
    First, we show that $\pi(\tilde{\Sigma}_2)$ contains a ball whose radius depends only on the distance between $\tilde{\Sigma}_1$ and $\tilde{\Sigma}_2$. Second, we prove that this ball descends to an embedded ball in $\Sigma_1$.

    Let $p \in \tilde{\Sigma}_1$ and $q\in \tilde{\Sigma}_2$ be the endpoints of $\tilde{\alpha}$ in $\tilde{\Sigma}_1$ and $\tilde{\Sigma}_2$ respectively. For any geodesic half-ray $\tilde\beta:[0,\infty)\to\tilde{\Sigma}_2$ with $\tilde\beta(0)=q$,  we give a lower bound for $r := \dist(p,\pi\circ\tilde \beta(\infty)))$.

Since $\tilde{\Sigma}_1$ and $\tilde \Sigma_2$ are at a positive distance from each other, it implies that $r$ should be finite.

    We consider a geodesic quadrilateral as shown in Figure \ref{fig:1} (a).   We cut it by a geodesic from $p$ to the point at infinity to obtain two ideal triangles $I$ and $\II$, with $r=\dist(p,\pi\circ\tilde \beta(\infty))$ being the length of the finite edge of {triangle}  $I$ and $a$ the length of the finite edge of {triangle}~$\II$, as in Figure \ref{fig:1} (b). Let $\phi$ be the angle at $p$ in $I$, so that the angle at $p$ in {triangle} $\II$ is $\frac{\pi}{2} - \phi$.
    
\begin{figure}[h]
    \centering
    \def\svgwidth{0.35\textwidth}
    %% Creator: Inkscape 1.1.2 (0a00cf5339, 2022-02-04), www.inkscape.org
%% Accompanies image file 'figure-1.pdf' (pdf, eps, ps)
\begingroup%
  \makeatletter%
  \providecommand\color[2][]{%
    \errmessage{(Inkscape) Color is used for the text in Inkscape, but the package 'color.sty' is not loaded}%
    \renewcommand\color[2][]{}%
  }%
  \providecommand\transparent[1]{%
    \errmessage{(Inkscape) Transparency is used (non-zero) for the text in Inkscape, but the package 'transparent.sty' is not loaded}%
    \renewcommand\transparent[1]{}%
  }%
  \providecommand\rotatebox[2]{#2}%
  \newcommand*\fsize{\dimexpr\f@size pt\relax}%
  \newcommand*\lineheight[1]{\fontsize{\fsize}{#1\fsize}\selectfont}%
  \ifx\svgwidth\undefined%
    \setlength{\unitlength}{371.50385338bp}%
    \ifx\svgscale\undefined%
      \relax%
    \else%
      \setlength{\unitlength}{\unitlength * \real{\svgscale}}%
    \fi%
  \else%
    \setlength{\unitlength}{\svgwidth}%
  \fi%
  \global\let\svgwidth\undefined%
  \global\let\svgscale\undefined%
  \makeatother%
  \begin{picture}(1,0.9122812)%
    \lineheight{1}%
    \setlength\tabcolsep{0pt}%
    \put(0,0){\includegraphics[width=\unitlength,page=1]{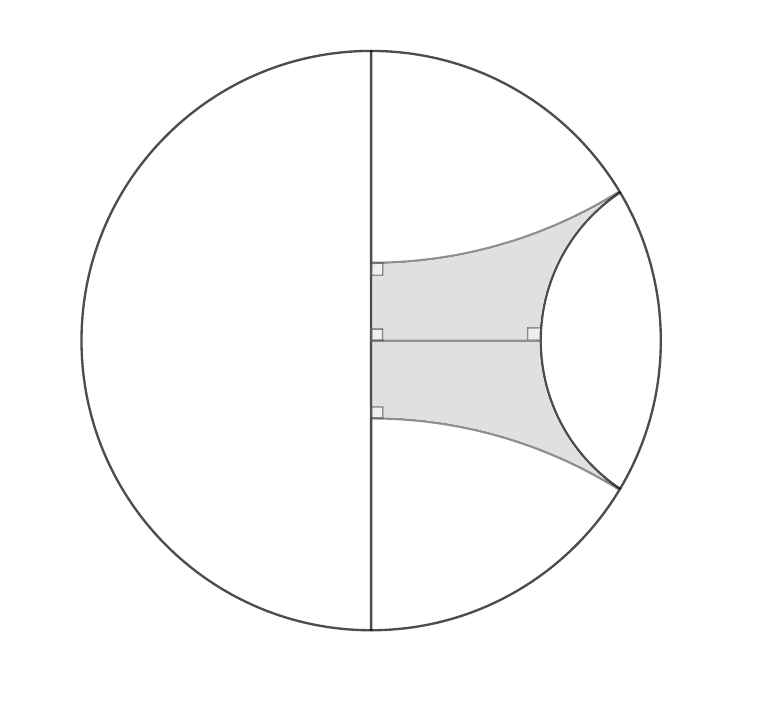}}%
    \put(0.40223412,0.66670459){\color[rgb]{0,0,0}\makebox(0,0)[lt]{\lineheight{1.25}\smash{\begin{tabular}[t]{l}$\tilde{\Sigma}_1$\end{tabular}}}}%
    \put(0.56472247,0.48497061){\color[rgb]{0,0,0}\makebox(0,0)[lt]{\lineheight{1.25}\smash{\begin{tabular}[t]{l}$\tilde{\alpha}$\end{tabular}}}}%
    \put(0.72483051,0.46228312){\color[rgb]{0,0,0}\makebox(0,0)[lt]{\lineheight{1.25}\smash{\begin{tabular}[t]{l}$\tilde{\Sigma}_2$\end{tabular}}}}%
    \put(0.43188148,0.40425377){\color[rgb]{0,0,0}\makebox(0,0)[lt]{\lineheight{1.25}\smash{\begin{tabular}[t]{l}$r$\end{tabular}}}}%
      \put(0.4351757,-0.03823197){\color[rgb]{0,0,0}\makebox(0,0)[lt]{\lineheight{1.25}\smash{\begin{tabular}[t]{l}(a)\end{tabular}}}}%
  \end{picture}%
\endgroup%

    \def\svgwidth{0.45\textwidth}
    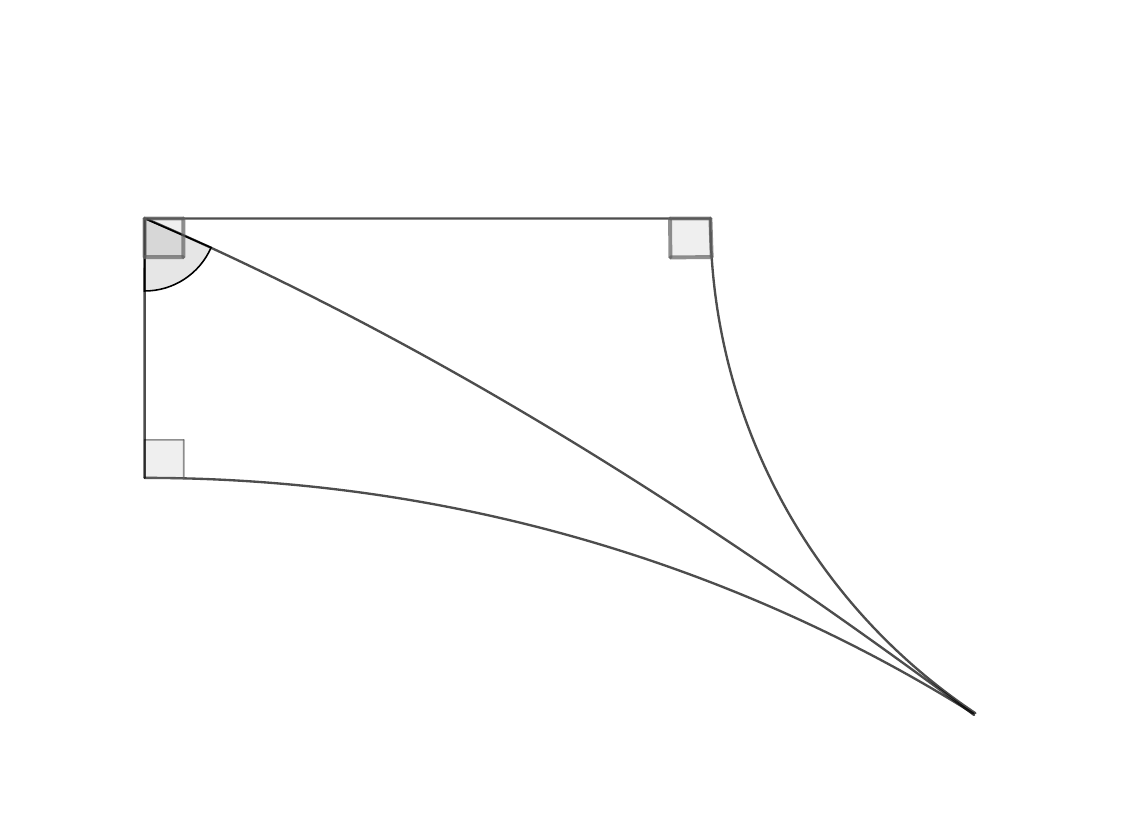
    \caption{(a) An illustration of the projection of $\tilde \Sigma_2$ under map $\pi$. (b) Triangles $I$ and $\II$.}
    \label{fig:1}
\end{figure}

    We apply  Ideal Triangle Comparison Lemma \ref{lem:triangle-comp} to triangle $I$ to obtain a right angle ideal triangle $\bar{I}$ in the hyperbolic plane $\mathbb{H}^2$ with side length $\bar{r} = r$ and angle $\bar{\phi}$ at $\bar{p}$ satisfying $\bar{\phi} \leq \phi$. Using the trigonometric identity $\sinh(\bar{r}) = \cot(\bar{\phi})$  in the hyperbolic plane \cite[Theorem 2.2.2]{Bus92}, we get
    \begin{align}\label{inq:I}
        \sinh(r) = \sinh(\bar{r}) = \cot(\bar{\phi}) {\geq} \cot(\phi).
    \end{align}

    Applying  Ideal Triangle Comparison Lemma \ref{lem:triangle-comp} to triangle $\II$, we have a right angle ideal triangle $\bar{\II}$ in the hyperbolic plane with a side length $\bar{a} \leq a$ and angle $\bar{\phi}'$ at $\bar{p}$ satisfying $\bar{\phi}' \leq \frac{\pi}{2} - \phi$. Similarly, by a trigonometric identity in the hyperbolic plane, we get 
    \begin{align}\label{inq:II}
        \sinh(a) \geq \sinh(\bar{a}) = \cot\left(\bar{\phi}'\right) \geq \cot\left(\frac{\pi}{2} - \phi\right) = \tan \phi.
    \end{align}

    Multiplying inequalities \eqref{inq:I} and \eqref{inq:II} gives us 
    \begin{align*}
        \sinh a \sinh r \geq 1
    \end{align*}
    and therefore
    \begin{align*}
        r \geq r(a) := \sinh^{-1}\left(\frac{1}{\sinh a}\right)=\log\coth\frac{a}{2}.
    \end{align*}

    Recall that the geodesic half-ray $\tilde\beta$  is orthogonal to {$\tilde \alpha$} and arbitrary. Therefore, $\pi(\tilde{\Sigma}_2)$ must contain a ball $B_{r(a)}$ in $\tilde{\Sigma}_1$ centered at $p$ of radius $r(a)$.
    
    We claim that this ball projects from the universal cover $\tilde{N}$ to an embedded ball of radius $r(a)$ in $\Sigma_1 \subset N$. Let $G$ be the group of deck transformation of $\tilde{N} \to N$. If $B_{r(a)}$ does not embed in $\Sigma_1$, there exists $\gamma \in G \setminus \{1\}$ such that $\gamma B_{r(a)} \cap B_{r(a)} \neq \emptyset$. Since $\Sigma_1$ is embedded, $\gamma$ must be in the stabilizer of $\tilde{\Sigma}_1$, i.e. $\gamma\tilde{\Sigma}_1 = \tilde{\Sigma}_1$. Since $\gamma\neq1$, we should have $\gamma \tilde{\Sigma}_2  \cap \tilde{\Sigma}_2=\emptyset$, see e.g. \cite[\textsection8]{BGS85}. However,  since $\emptyset\neq\gamma B_{r(a)} \cap B_{r(a)} \subset\pi(\gamma \tilde{\Sigma_{2}} 
    \cap \tilde{\Sigma}_2)$, we must have $\gamma \tilde{\Sigma}_2  \cap \tilde{\Sigma}_2 \neq \emptyset$ and we get a contradiction. 
\end{proof}

\begin{proof}[Proof of Theorem \ref{thm:tubeII}]
    We first show the existence of the tubular neighborhood around $\Sigma$. If $\exp^\perp : (-\tube_{n,\kappa}(A), \tube_{n,\kappa}(A)) \times \Sigma \to N$ is not a diffeomorphism, then we can find a geodesic arc $\gamma : [0,T] \to N, \gamma(t) = \exp^\perp(t,p)$ for some $p \in \Sigma$ and some $0 < T < 2 \tube_{n,\kappa}(A)$, such that $\gamma(T) \in \Sigma$ and is orthogonal to $\Sigma$ at $\gamma(T)$. Since $$\length(\gamma)=T <2 \tube_{n, \kappa}(A),$$  by  Proposition \ref{prop:r(a)} and the fact that $r(\cdot)$ is a decreasing function, there exists an embedded ball $B_{r(T)}$ in {$\Sigma$} of radius $r(T) > r(2 \tube_{n,k}(A))$.

    By the G\"unther-Bishop volume comparison theorem \cite[Theorem III.4.2]{chavel_2006}, we have
    \begin{align*}
        \vol(B_{r(T)}) \geq V_{n-1,\kappa}(r(T))) > V_{n-1,\kappa}(r(2\tube_{n,k}(A))).
    \end{align*}
    Recall that $\tube_{n,\kappa}(A) = \frac{1}{2}(V_{n-1,\kappa} \circ r)^{-1}(A)$ and  we get the contradiction
    \begin{align*}
        A \geq \vol(B_r) > (V_{n-1,\kappa} \circ r)(V_{n-1,\kappa} \circ r)^{-1}(A) = A.
    \end{align*}
    Hence, $\exp^\perp : (-\tube_{n,\kappa}, \tube_{n,\kappa}) \times \Sigma \to N$ is a diffeomorphism.

    We now prove that the tubular neighborhood of any two disjoint embedded totally geodesic hypersurfaces, say $\Sigma_1$ and $\Sigma_2$, are disjoint. If $d$ is the length of the shortest geodesic arc from $\Sigma_1$ to $\Sigma_2$ and orthogonal to both, we have by Proposition \ref{prop:r(a)}
    \begin{align*}
        (V_{n-1,\kappa} \circ r)(d) \leq \vol(\Sigma_j) \quad \text{for } j = 1,2
    \end{align*}
    and thus
    \begin{align*}
        d &\geq \frac{1}{2}(V_{n-1,\kappa} \circ r)^{-1}(\vol(\Sigma_1)) + \frac{1}{2}(V_{n-1},\kappa \circ r)^{-1}(\vol(\Sigma_2)) \\
        &= \tube_{n,\kappa}(\vol(\Sigma_1)) + \tube_{n,\kappa}(\vol(\Sigma_2)).
    \end{align*}
    We conclude the corresponding tubular neighborhoods must be disjoint.
\end{proof}

\begin{remark}\label{rem:wn}\begin{itemize}
    \item[(a)] In the proof of Theorem \ref{thm:tubeII}, we use the G\"unther-Bishop  volume comparison theorem on $\Sigma$ for which only the upper bound for sectional curvature of $\Sigma$ is needed. Hence, if the upper bound of the sectional curvature of the totally geodesic embedded hypersurface is $-\alpha^2$, $\alpha\in(\kappa,1]$, then $\kappa$ can be replaced with $\alpha$. In particular when  $\Sigma$ has constant sectional curvature $-1$, we can improve the width of tubular neighborhood of $\Sigma$ to $\tube_n(A)$. 
   \item[(b)] When an embedded closed totally geodesic hypersurface $\Sigma$ of volume $A$ is separating, i.e. $M\setminus\Sigma$ has two connected components, then we can improve the width of the tubular neighborhood of $
   \Sigma$ to $\tube_{n,\kappa}(\frac{A}{2})$.
    \end{itemize}
\end{remark}

 We now study the growth rate of the volume of the tube of width $\tube_{n,\kappa}(A)$ around an embedded totally geodesic hypersurface $\Sigma$. We denote it by $\Vtube_{n,\kappa}(\Sigma)$. When $\kappa=1$, we omit $\kappa$, $\Vtube_n(\Sigma)=\Vtube_{n,\kappa}(\Sigma)$.  In the hyperbolic setting, we refer to \cite[Lemma 3.1]{Bas94} for some explicit formulas for $\tube_n(x)$ and $\Vtube_n$ and their asymptotic behavior. 
 
In the Fermi coordinates based on $\Sigma$ (as given by $\exp^\perp$ in Theorem \ref{thm:tubeII}, see e.g. \cite[III.6]{chavel_2006}),  $\Vtube_{n,\kappa}(\Sigma)$ is given by
      \begin{align*}
          \Vtube_{n,\kappa}(\Sigma) =  \int_{-\tube_{n,\kappa}(A)}^{\tube_{n,\kappa}(A)} \int_{\Sigma}A(t,x)  \dd v_{\Sigma} \dd t,
      \end{align*}
      where $A(t,x) = |\det (\dd \exp^\perp(t,x))|$ with $\dd \exp^\perp$ the differential of the exponential map defined in Theorem \ref{thm:tubeII}. 
  Since the  sectional curvature of $N$ is upper bounded by  $-\kappa^2$, by a version of the G\"unther-Bishop volume  comparison theorem for tubes (see \cite[Lemmas 8.22 and 8.24]{G03} and \cite{HK78}), we have 
      \begin{align*}
          \cosh^{n-1}(\kappa t)\le A(t,x)\le \cosh^{n-1}(t), \qquad t\in(-\tube_{n,\kappa}(A), \tube_{n,\kappa}(A)).
      \end{align*}
Therefore,
      \begin{equation}\label{inq:vgrowth}
    2A \int_{0}^{\tube_{n,\kappa}(A)}  \cosh^{n-1}(\kappa t)\dd t\le \Vtube_{n,\kappa}(\Sigma)\le \Vtube_n(\Sigma)=  2A \int_{0}^{\tube_{n}(A)}  \cosh^{n-1}( t)\dd t. 
      \end{equation}
We use the inequality $\tube_{n,\kappa}(A)\le \tube_{n}(A)$ to get the second inequality. \\
      The following lemma gives a lower bound for the width function  $\tube_{n,\kappa}$. This lower bound will be also used in Section \ref{sec:stekbound} to get bounds on the Steklov eigenvalues.  
\begin{lemma} \label{lem:cn}
    For any integer $n \geq 3$ and a positive constant $C$ there exists a positive constant   $c_{1}(n,\kappa,C)$ such that for every $x \ge C$ the following inequality holds.
    \begin{align}\label{inq:cn}
        \tube_{n, \kappa}(x) &\geq  c_{\countc}(n,\kappa,C)  x^{-\frac{1}{\kappa(n-2)}}.
    \end{align}
\end{lemma}
\begin{remark}\label{rem:lowerv}
The domain of $\tube_{n,\kappa}$ is usually the volume of a $(-\kappa^2)$–\,pinched negatively curved $(n-1)$–\,manifold which is bounded below by $V_{n-1,\kappa}(\mu(n-1))$, where $\mu(n-1)$ is the Margulis constant, see inequality \eqref{inq:lowerv}. Hence,  we can take $C$  equal to $V_{n-1,\kappa}(\mu(n-1))$ in the above lemma.  
\end{remark}

\begin{proof}[Proof of Lemma \ref{lem:cn}]
    We first find a lower bound for $V_{n-1, \kappa} (r(2y))$ which is the inverse of $\tube_{n,\kappa}$. 
    For any $y$ such that $\coth {y} > 2^{1/2\kappa}$, we have
    \begin{align*}
        V_{n-1, \kappa} (r(2y)) &= \omega_{n-2} \int_0^{\log \coth y} \left(\frac{\sinh(\kappa t)}{\kappa}\right)^{n-2} \dd t \\
        &\geq \omega_{n-2} \int_{\frac{\log 2}{2\kappa}}^{\log \coth y} \left(\frac{\sinh(\kappa t)}{\kappa}\right)^{n-2} \dd t \\
        &\geq \omega_{n-2} \int_{\frac{\log 2}{2\kappa}}^{\log \coth y} \left(\frac{e^{\kappa t}}{4\kappa}\right)^{n-2} \dd t \\
        &= \frac{\omega_{n-2}}{(n-2)4^{n-2} \kappa^{n-1}} \left( (\coth y)^{\kappa(n-2)} - 2^{\frac{n-2}{2}}\right),
    \end{align*}
     where $\omega_{n-2}$ is the volume of the $(n-2)$–\,dimensional unit sphere in $\mathbb{R}^{n-2}$. \\
  Applying $\tube_{n,\kappa}$ to both sides of the inequality, setting the last term equal to $x$, and solving for $y$ using that $\mathrm{arcoth}(z) = \frac{1}{2} \log\left(1 + \frac{2}{z - 1}\right)$, gives
\[\tube_{n, \kappa}(x)\ge y=\log\left(1+ \frac{2}{\left(\frac{(n-2)4^{n-2}}{\omega_{n-2}}\kappa^{n-1} x + 2^{\frac{n-2}{2}}\right)^{\frac{1}{\kappa(n-2)}} - 1}\right).\]
Since $\tube_{n,\kappa}$ is a decreasing function, the lower bound holds for all $x> 0$.
    Taking into account the asymptotic behavior of the right-hand side as $x$ goes to infinity and the fact that it is decreasing in $x$,  there exists $x_0>0$ such that for all $x\ge x_0$ we have

     \begin{align*}
        \tube_{n, \kappa}(x) &\geq c_{\countc}(n,\kappa)  x^{-\frac{1}{\kappa(n-2)}}.
    \end{align*}
    To get the desired {inequality \eqref{inq:cn}}, we take $$c_1(n,\kappa)=\min\{\tube_{n,\kappa}(C)C^{\frac{1}{\kappa(n-2)}},c_2(n,\kappa)\},$$
    where $C$ is the lower bound for $x$ given in the statement of the lemma.
\end{proof}
\begin{lemma}\label{lem:vol}
    Let $\Sigma$ be a closed totally geodesic hypersurface of volume $A$ in $N$. Then
    \[c_{\countc}(n,\kappa)A^{1-\frac{1}{\kappa(n-2)}}\le \Vtube_{n,\kappa}(\Sigma)\le \Vtube_n(\Sigma)\le c_{\countc}(n)A^{1-\frac{1}{n-2}}.\]
  In particular,  $\displaystyle{\lim_{A\to \infty}\Vtube_{n,\kappa}(\Sigma)}=\infty$  when $\kappa\in (\frac{1}{n-2},1]$, and  $\displaystyle{\lim_{A\to \infty}\Vtube_{n,\kappa}(\Sigma)\ge c_5(n,\kappa)}$  for $\kappa=\frac{1}{n-2}$.

    \end{lemma}
\begin{proof}
    By Lemma \ref{lem:cn} and inequality \eqref{inq:vgrowth}, we get
    \begin{align*}
 \Vtube_{n,\kappa}(\Sigma)
    &\ge 2A\int_0^{\tube_{n,\kappa}(A_j)}\frac{e^{(n-1)\kappa t}}{2^{n-1}}\dd t\\
    &= \frac{A(e^{(n-1)\kappa \tube_{n,\kappa}(A)}-1)}{\kappa(n-1) 2^{n-2}}\\ &\ge \frac{A\tube_{n,\kappa}(A)}{2^{n-2}}\\
    &\ge c_{\countc}(n,\kappa)A^{1-\frac{1}{\kappa(n-2)}}\,.
\end{align*}
\end{proof}
Lemma \ref{lem:vol} does not give any nontrivial lower bound for $\lim_{A\to\infty}\Vtube_{n,\kappa}(\Sigma)$ when $\kappa\in (0, \frac{1}{n-2})$. As for the upper bound, as it is shown in \cite{Bas94}, we have
\[\lim_{A\to \infty}\Vtube_n(\Sigma)=\begin{cases}
    \infty &n\ge 4,\\
    2\pi&n=3.
\end{cases}\]
 
\begin{remark}
    In the hyperbolic case, the growth rate of $\Vtube_n$ can immediately give a linear upper bound for the number of boundary components $b$ in terms of $V$ as in \cite{Bas94}.  
    The linear upper bound for $b$ in terms of $V$ is also a direct  consequence of the following inequality  proved by Zeghib \cite[\S 4.4]{Zeg91}, and Miyamoto  \cite{miy94}. 
    \begin{equation}\label{eq:zm}
    c_{\countc}(n)\vol(\partial M)\le \vol(M). \end{equation}
    
     {However, the growth rate $\Vtube_{n,\kappa}$ does not lead such bound when $\kappa\in(0,\frac{1}{n-2})$.} In the setting of the pinched negatively curved manifolds, Zeghib's result \cite[\S 4.7]{Zeg91}  gives a linear upper bound for $b$ in terms of $V$ since \eqref{eq:zm} remains true with a constant depending on the dimension and the pinching constant, see inequality \eqref{inq:az}.  \\
   
\end{remark}

\section{Steklov eigenvalue bounds}\label{sec:stekbound}

\setcounter{countc}{0}
\subsection{Bounds for $\sigma_b(M)$} \label{sec:stekbound-b}

We first state the comparison theorem between the Steklov eigenvalues $\sigma_k(M)$ and the Laplace eigenvalues $\lambda_k(\partial M)$. It is a consequence of a result  in  \cite{CGH20}, together with applying tubular neighborhood theorem~\ref{thm:tubeII} and Lemma \ref{lem:cn}. 

{We recall that we use Convention \ref{convention} regarding constant coefficients throughout this section.}
 
\begin{proposition}\label{prop:CGH}
Let $M$ be a compact $(-\kappa^2)$–\,pinched negatively curved $n-$manifold with totally geodesic boundary. Let $\lambda_k(\partial M)$ be the $k$th Laplace eigenvalue of $\partial M$ and $A$ be the maximum volume of its boundary components.  Then  we have 
\begin{equation}\label{lowerb}
\sigma_k(M)\ge\frac{\lambda_k(\partial M)}{c_{\countc}(n,\kappa)A^{\frac{1}{\kappa(n-2)}}+\sqrt{\lambda_k(\partial M)}}\,,
    \end{equation}
and 
\begin{equation}\label{upperb}
\sigma_k(M) \le c_{\countc}(n,\kappa)A^{\frac{1}{\kappa(n-2)}}+\sqrt{\lambda_k(\partial M)}\,.
\end{equation}
\end{proposition}

\begin{proof}
We use the following  inequalities between $\sigma_k(M)$ and $\lambda_k(\partial M)$ established in \cite[Theorem 7]{CGH20}.
\begin{eqnarray*}
\lambda_k(\partial M)\le \sigma_k(M)^2+\alpha \sigma_k(M);\\
\sigma_k(M)\le \beta+\sqrt{\lambda_k(\partial M)},
\end{eqnarray*}
where $\alpha=1+\frac{1}{\roll(M)}$ and $\beta=\frac{1}{\roll(M)}+n$.
Here, $\roll(M)$ of $M$ denotes the distance between $\partial M$ and its cut locus and is called the \textit{rolling radius}. {In other words, it is the largest $T$ such that the map $\exp^\perp : [0, T) \times \del M \to M$ is injective}.
We can view the first inequality as a second-order polynomial in terms of $\sigma_k$. Determining the range of $\sigma_k$ for which the inequality holds leads to
\[\sigma_k(M)\ge\frac{2\lambda_k(\partial M)}{{\alpha}+\sqrt{\alpha^2+4\lambda_k(\partial M)}}\ge \frac{\lambda_k(\partial M)}{{\alpha}+\sqrt{\lambda_k(\partial M)}}.\]
Now, by the tubular neighborhood theorem in the hyperbolic setting   and Lemma \ref{lem:cn}, we have 
\[\roll(M)\ge \tube_{n,\kappa}(A)\ge c_{\countc}(n,\kappa)A^{-\frac{1}{\kappa(n-1)}}.\]
It gives an upper bound for $\alpha$ and $\beta$ and completes the proof. 
\end{proof}
In \cite{CGH20,DL91},  other cases where the lower bounds on $\roll(M)$ can be estimated are discussed.  However, those cases do not cover the setting of negatively curved manifolds with totally geodesic boundary. \\

Having Proposition \ref{prop:CGH}, we can use well-known geometric bounds for the  Laplace eigenvalues on closed negatively curved manifolds to obtain geometric lower bounds for $\sigma_k(M)$, $k\ge b$. 
\begin{proof}[Proof of Theorem \ref{thm:Schoencounterpart}] It is  a consequence of Proposition \ref{prop:CGH}  together with the result of Schoen \cite{Schoen82} which gives a lower bound for the first nonzero Laplace eigenvalue of a closed $(-\kappa^2)$–\,pinched negatively curved $n$–\,manifold of dimension at least 3.

The right-hand side of \eqref{lowerb} is monotone increasing in $\lambda_k(\partial M)$. Thus, any lower bound for $\lambda_b(\partial M)$ will give a lower bound for $\sigma_b(M)$. Let $\Sigma_j$, $1\le j\le b$, be the boundary components of $M$. 
Hence,
\begin{align}\label{inq:schoen1}
    \lambda_b(\partial M)=\min_j\lambda_1(\Sigma_j) \geq \min_j\frac{c_{\countc}(n,\kappa)}{\vol(\Sigma_j)^2} \geq \frac{c_{4}(n,\kappa)}{A^2}. 
\end{align}
where the first inequality is a consequence of the result of Schoen in \cite{Schoen82}. 
We conclude by substituting inequality \eqref{inq:schoen1} into inequality \eqref{lowerb}. 
\end{proof}

\begin{remark}\begin{itemize}
    
\item[a)] The lower bound in Theorem \ref{thm:Schoencounterpart} can be expressed in terms of the volume of $M$ using the following inequality 
\[\vol(M)\ge c_{\countc}(n,\kappa)\vol(\partial M)\]
which is proved by Zeghib \cite{Zeg91}, and independently by Miyamoto \cite[Theorem 4.2]{miy94} when $M$ is a  hyperbolic manifold.
However, replacing $A$ by $V$ gives a weaker lower bound, see Example \ref{ex:smalleigenvalue}. \\
\item[b)]
Buser, Colbois, and Dodziuk \cite[Theorem 3.1]{BCD93} obtained a lower bound for the first Laplace eigenvalue on negatively curved manifolds which is independent of $\kappa$ when $n\ge4$. They also gave examples showing that the lower bound cannot be independent of $\kappa$ when $n\le3$. It is not clear whether a counterpart of Buser-Colbois-Dodziuk's result for $\sigma_b(M)$ holds when the dimension of the boundary is at least four. \\ 
\item[c)] We can obtain the following upper bound for $\sigma_b(M)$ 
\begin{eqnarray*}
   \sigma_b(M)\le   c_{\countc}(n,\kappa)A^{\frac{1}{\kappa(n-2)}}.
\end{eqnarray*}
by substituting Buser's inequality \cite[Theorem 6.2]{Bus82}: {$$\lambda_b(\partial M)=\min_j\lambda_1(\Sigma_j)\le \frac{(n-1)^2}{4}+c_{\countc}(n)\left(\frac{1}{\max_j\vol(\Sigma_j)}\right)^{\frac{2}{n-1}},$$}
{where $\Sigma_j$ are  connected components of $\partial M$,}
in inequality \eqref{upperb}. We refer to \cite[Theorem 4.1]{Has11} for a different type upper bound for $\sigma_k(M)$.\\
When $A$ tends to infinity, the $k$th Laplace eigenvalue of each boundary component remains uniformly bounded for all $k$ less than the volume of that boundary component.     It is an intriguing question whether $\sigma_k(M)$,  $k\le A$, has the same behavior as the Laplace eigenvalues and remains bounded when $A\to \infty$.
\end{itemize}
\end{remark}

We proceed with stating Steklov eigenvalue bounds when $M$ is a compact pinched negatively curved 3-manifold. We need to recall some of the well-known bounds for the Laplace eigenvalues on closed surfaces.

Let $\Sigma$ be a  closed $(-\kappa^2)$–\,pinched surface of genus $g$. For $1\leq k\leq 2g-3$, let $\Ck_k$ be the set of all curves $\gamma$ consisting of a disjoint union of simple closed geodesics dividing $\Sigma$ into $k+1$ components. Let us  define
\begin{equation}\label{def:l_k}
    \ell_k(\Sigma):=\inf_{\gamma\in\Ck_k} \length(\gamma).
\end{equation}
Schoen, Wolpert, and Yau  \cite{SWY80} showed that \begin{equation}\label{inq:swy}
    c_{\countc}(g)\kappa^{3}\ell_k(\Sigma) \leq \lambda_k (\Sigma)\leq c_{\countc}(g)\ell_k(\Sigma),\quad 1\le k\le 2g-3
\end{equation} 
and 
\begin{equation}\label{inq:swy2}
c_{8}(g)\kappa^2\le\lambda_{2g-2}(\Sigma)\le c_{9}(g).
\end{equation}
Note that $c_{9}(g)$ can be replaced by a universal constant thanks to Buser's inequality \cite{Bus82}:
$$\lambda_k(\Sigma)\le\frac{1}{4}+c_{\countc}\frac{k}{2g-2}.$$

We refer to   \cite{burger88,burger90,OS09,GR19,CR22} for some improvement of inequalities \eqref{inq:swy} and \eqref{inq:swy2} for hyperbolic surfaces.

Similar to the proof of Theorem \ref{thm:Schoencounterpart}, 
substituting Schoen-Wolpert-Yau inequalities \eqref{inq:swy} and \eqref{inq:swy2} in the inequalities given in  Proposition  \ref{prop:CGH} and using the Gauss-Bonnet theorem to estimate the volume, we obtain the following inequalities. 
\begin{proposition}\label{SWY3}
Let $M$ by a $(-\kappa^2)$–\,pinched negatively curved 3-manifold  with a connected totally geodesic boundary $\Sigma$ of genus $g$. Then there exist positive constants $c_{\countc}(g,\kappa)$, $c_{\countc}(\kappa)$  and $c_{\countc}$ such that 
\begin{equation}\label{inq:1}
   c_{11}(g)\ell_k(\Sigma) \le \sigma_k(M) \leq c_{12}(\kappa)g^{\frac{1}{\kappa}}+c_{13}\sqrt{\ell_k(\Sigma)} , \qquad 1 \leq k \leq 2g-3,
\end{equation}
and 
\begin{equation}
c_{11}(g,\kappa)\le \sigma_{2g-2}(M) \leq c_{12}(\kappa)g^{\frac{1}{\kappa}}+c_{13}.
\end{equation}
\end{proposition}
\noindent It is an interesting question whether the upper bound can be independent of $g$. 

In the following section, we prove an alternative lower bound for $\sigma_1(M)$ wherein a function of the total volume of $M$ replaces the role of $\ell_1(\Sigma)$ in \eqref{inq:1}.

When $\partial M$ is not connected,   
let $\Sigma_j$ with genus $g_{j}$, $1\le j\le b$, denote the boundary components of $M$,  $\ell_1(\partial M):= \min_j \ell_1(\Sigma_j)$, and $g:=\max_j g_j$. We have 
\[\sigma_b(M)\ge c_{\countc}(g,\kappa)\ell_1(\partial M).\]

\subsection{Lower bound for $\sigma_1(M)$}\label{subset:sigma1}
\setcounter{countc}{0}
Let $M$ be a $(-\kappa^2)$–\,pinched $n$–\,manifold, where $\kappa\in(0,1]$ and $n\ge 3$, and having  totally geodesic boundary.  Let $b$ be the number of boundary components, and $V$ and $A$ denote the volume of $M$ and the maximum volume of the boundary components respectively. We use $\ostube^{\delta}(\partial M)$, $\delta\in (0,\tube_{n,\kappa}(A)]$, to denote the $\delta$–\,tubular neighborhood of $\partial M$. When $\delta=\tube_{n,\kappa}(A)$, we simply denote it by $\ostube(\partial M)$.
\begin{theorem}\label{thm:specgap}
Under the notations and assumptions stated above, {we have}  
  \begin{equation}\label{eq:pV}
    \sigma_1(M)\ge \frac{c_{1}(n,\kappa)}{bA^{\frac{2n}{\kappa(n-2)}}V\, V_1 },\end{equation}
where {$c_{\countc}(n,\kappa)$ is a positive  constant depending only on $n$ and $\kappa$} such that u $V_1=\vol(\bM\cup \ostube(\partial M))\le V$.
    
\end{theorem}
\noindent It is obvious that Theorem \ref{thm:specgap-intro} follows from the above theorem.\\

In Section \ref{sec:smallstek}, we give an example of a sequence of manifolds in which $b$, $A$, and $V_1$ are bounded for each element in the sequence, while $V$ tends towards infinity. We show that $\sigma_1(M)$ goes to zero with a rate proportional to $\frac{1}{V}$. It, in particular, shows that the power of $V$ in inequality~\eqref{eq:pV} is optimal. \\

To prove Theorem \ref{thm:specgap}, we invoke the argument used by Dodziuk and Randol \cite{DR86,Dod}  to give an alternative proof of Schoen–Wolpert–Yau's inequality \cite{SWY80} in dimension two and  Schoen's inequality \cite{Schoen82} in higher dimensions for closed negatively curved manifolds. We modify and adapt their ideas to the Steklov problem. The fact that we are dealing with manifolds with boundary introduces certain technical challenges, and the tubular neighborhood theorem \ref{thm:tube} proves crucial in overcoming these technical obstacles. 

We begin by stating a key lemma that provides a lower bound for the first eigenvalue of a Steklov-Dirichlet problem `near' the boundary. 

 Let $\Sigma $ be a connected component of $\partial M$ of volume $A$, and $\ostube^\delta(\Sigma) $ the $\delta$–\,tubular neighborhood of $\Sigma$ with $\delta\le\tube_{n,\kappa}(A)$. Let $\Omega=\ostube^\delta(\Sigma)\cup M_{\thin,\Sigma}^{\delta}$, where 
 \begin{equation}\label{mthinksigma}
{M_{\thin,\Sigma}^{\delta}}:=\{T_\gamma\in M_{\thin}^{\delta}: \text{$T_{\gamma}$ is a type II tube and $\gamma\subset \Sigma$}\}.
 \end{equation}
 We consider the mixed Steklov-Dirichlet eigenvalue problem on $\Omega$:
\begin{align*}
    \begin{cases}
        \Delta f = 0 & \text{in } \Omega, \\
        \del_{\nu} f = \sigma^D f & \text{on } \Sigma, \\
        f = 0 & \text{on } \partial \Omega\setminus\Sigma.
    \end{cases}
\end{align*}
\begin{lemma}\label{lem:stekdir}
The first Steklov-Dirichlet eigenvalue $\sigma_1^D(\Omega)$  is bounded below by $\frac{\kappa(n-1)}{2^{n-1}}$.
 \end{lemma}
  \begin{proof}
     Let  $f$ be the eigenfunction corresponding to $\sigma_1^D(\Omega)$. Let us consider the Fermi coordinates on $\Omega$ based on $\Sigma$, given by $(t, x)$ where $t$ is the distance from $\Sigma$ and $x$ is the position in $\Sigma$ (see \cite[III.6]{chavel_2006} for details). In these coordinates we have 
     \begin{align*}
          \int_\Omega |\nabla f|^2 \dd v
          &\geq \int_\Omega (\del_t f)^2 \dd v = \int_\Sigma \int_0^{r_x} (\del_t f)^2  A(t,x) \dd t \dd v_\Sigma, 
      \end{align*}
       where $r_x$ is the distance between $x$ and the $\partial \Omega \setminus \Sigma$, $A(t,x) = |\det (\dd \exp^\perp)|$ with $\dd \exp^\perp$ the differential of the exponential map of the normal bundle of $\Sigma$, and $\dd v_\Sigma$ is the volume element on $\Sigma$. 
      Since the sectional curvature is bound above  by $-\kappa^2$, we have the following inequality \cite{HK78,G03}
      \begin{align*}
          A(t,x) = |\det (\dd \exp^\perp)| \geq \cosh^{n-1}(\kappa t).
      \end{align*}
     Using the Cauchy-Schwarz inequality and the above inequality, we get
      \begin{align*}
          \int_\Omega |\nabla f|^2 \dd v
          &\ge  \int_\Sigma \int_0^{r_x} (\del_t f)^2  A \dd t \dd v_\Sigma \\
          &\geq \int_\Sigma \left(\frac{1}{\int_0^{r_x} \frac{1}{A(t,x)} \dd t} \right) \left(\int_0^{r_x} \del_t f \dd t \right)^2 \dd v_\Sigma \\
          &\ge \int_\Sigma \frac{1}{\int_0^{r_x} {\frac{1}{\cosh^{n-1}(\kappa t)}} \dd t} f(0,x)^2 \dd v_\Sigma \\
           &\geq \frac{1}{\int_0^\infty \frac{1}{\cosh^{n-1}(\kappa t)} \dd t} \int_\Sigma f(0,x)^2 \dd v_\Sigma.
      \end{align*}
      
      We conclude 
      \begin{align*}
          \sigma_{1}^D(\Omega) &= \frac{\int_\Omega |\nabla f|^2 dv}{\int_\Sigma f^2 dv_\Sigma} 
          \geq \frac{1}{\int_0^\infty \frac{1}{\cosh^{n-1}(\kappa t) \dd t}} 
          \geq \frac{\kappa (n-1)}{2^{n-1}}.
      \end{align*}

\end{proof}
The second key lemma is an inequality for harmonic 1-forms.   We state a special case of this result, which will be used in the proof of Theorem \ref{thm:specgap}. For a general version of this lemma for functions, we refer to \cite{CGT}. It leads to a version for  1-forms as stated in \cite[page 32]{Dod}, see also \cite{DR86}.
\begin{lemma}\label{lem:sobinq} Let $\varphi$ be a harmonic function on a subset $\Omega$ of a closed $(-\kappa^2)$–\,pinched negatively curved $n$–\,manifold $N$. Let $\g$   denote the Riemannian metric on $N$.
For any (metric) ball $B\subset \Omega$ centered at $x\in \Omega$ of radius $0<r<\inj_{\g}(x)$, the following holds.

\begin{equation*}
|\nabla_{\g}\varphi(x)|\le  C{r^{-\frac{n}{2}}}\left(\int_B|\nabla_{\g} \varphi|^2\dd v_{\g}\right)^{\frac{1}{2}},
\end{equation*}
where $C$ is a  constant depending only  on $n$ and the $C^m(B)$–\,norm of the Riemann curvature tensor $\text{Rm}_{\g}$ on $B$ for some $m$ depending only on $n$.
\end{lemma}

\begin{remark}\label{rem:sobolevinq}
 As in Dodziuk \cite{Dod}, we want the constant $C$ to depend only on the dimension. This can be done using the results of Bemelmans, Min-Oo, and Ruh in   \cite{BMR}, and  Dodziuk in \cite{Dod}.
    More precisely,  for any $\delta>0$ { and $m\in \mathbb{N}$}, there exist a metric ${\g}_{\delta}$ on $N$ and positive constants $c_{\countc}(n,m,\delta)$ and $c_{\countc}(n,{m},\delta)$ such that $$(1-\delta)\g\le \g_\delta\le(1+ \delta)\g\,;\quad\quad \|\text{Rm}_{\g_\delta}\|_{C^m(N)}\le c_{2}(n,{m},\delta)\,;$$
     $$\inj_{\g_\delta}(x)\ge c_{3}(n,{m},\delta)>0,\qquad \forall x\in N_{\thick},$$
     where, $N_{\thick}$ is defined with respect to $g$. 
     
   \noindent  Let $\g_{\circ}=\g_{1/2}$. Thus by the variational characterisation \eqref{eq:variational characterisation} of $\sigma_1$ we get
     \[c_{\countc}(n)\sigma_1(M,\g_\circ) \le \sigma_1(M,\g)\le c_{\countc}(n)\sigma_1(M,\g_\delta).\]
     The same inequality holds if we replace $\sigma_1$ by $\sigma_1^D$.
      Therefore, for $\Omega$ the domain in Lemma~\ref{lem:stekdir}, we have 
      \[\sigma_{1}^D(\Omega,\g_{\circ})\ge c_{\countc}(n,\kappa).\]

\end{remark}
\noindent We now proceed with the proof of Theorem \ref{thm:specgap}.
\begin{proof}[Proof of Theorem \ref{thm:specgap}] 
First note that without loss of generality, we can assume the manifold to be orientable. Indeed if not, one can consider its orientable double cover and, since the Steklov spectrum of a manifold is a subset of the spectrum of its cover, lower bounds obtained for the orientable double cover provide lower bounds for the original non-orientable manifold.

Let $\g_\circ$ be as in Remark \ref{rem:sobolevinq}. It is enough to prove the theorem for $\sigma_1(M,\g_\circ)$. Let  $\varphi$ be a $\sigma_1(M,\g_{\circ})$–\,Steklov eigenfunction with $\|\varphi\|_{L^2(\partial M,\g_\circ)}=1$. 
The idea of the proof is to show that $\sqrt{\sigma_1(M,\g_\circ)}$ controls the variation of $\varphi$ on the thick part of $M$ away from the boundary. Hence if $\sigma_1(M,\g_\circ)$ is small, $\varphi$ is nearly constant there. Then two cases arise: if $\varphi$ is large, we shall see it cannot vanish in the thick part away from the boundary union with type I tubes and hence must have a nodal domain in a neighborhood of the boundary, allowing us to use Lemma \ref{lem:stekdir} to conclude; if $\varphi$ is small in this thick part, then we can use it to construct a test function for the Steklov-Dirichlet problem on a neighborhood of a boundary component and again use Lemma \ref{lem:stekdir} to conclude. 

Throughout the proof, only, the $\epsilon$–\,thick-thin decomposition of $M$ is given for the original metric $\g$ on $M$. Otherwise, we always use metric $\g_\circ$. We begin by showing that for some $\epsilon > 0$, to be fixed later, and for any point $x\in M^\epsilon_{\thick}$, positioned at a distance from the boundary of $M$, the magnitude of the gradient of $\varphi(x)$ is bounded above by $\sigma_1(M,\g_\circ)$. 

We now apply Lemma \ref{lem:sobinq}  to $\varphi$. 
 We get that for any ball $B\subset M$ centered at $x\in M$ and of radius $0<r<c_{\countc}(n)\le \mu(n)$, where $c_7(n):={c_3(n,1/2)}$ is a lower bound on $\inj_{\g_\circ}(x)$ given in Remark \ref{rem:sobolevinq}, we have
\begin{equation*}
    \|\nabla_{\g_\circ} \varphi\|_{\infty, B/2 }\le {c_{\countc}(n)}{r^{-\frac{n}{2}}}\left(\int_B|\nabla_{\g_\circ}\varphi|^2\right)^{1/2},
\end{equation*}
where  $B/2$ is the ball concentric with $B$ and of half the radius of $B$, and  Here, $\|\cdot\|_{\infty, B/2}$ denotes the sup norm on $B/2$.

    We can choose $\epsilon\le \mu(n)$  small enough such that there are no type III tubes, i.e. each connected component of $M^{2\epsilon}_{\thin}$ is a tube based on a unique closed simple geodesic either in the interior of $M$ or on $\partial M$ (see Section \ref{sect: preliminaries}). This can be achieved by choosing 
    \begin{equation}\label{def:epsilon}
        \epsilon:=c_{\countc}(n)\tube_{n,\kappa}(A),\end{equation}
    where $c_{9}(n)$ is a positive constant such that  $$c_{9}(n)\tube_{n,\kappa}(A)<\min\left\{\tube_{n,\kappa}(A),c_{7}(n),\frac{\mu}{5}\right\},$$ 
    
      Let $M_1:=M_{\thick}^\epsilon\setminus\ostube^\epsilon(\partial M)$. {Note that for any $x\in M_1$ $$\frac{\epsilon}{\sqrt{2}}\le\frac{1}{\sqrt{2}} \dist_{\g}(x,\partial M)\le \dist_{\g_\circ}(x,\partial M).$$} Then any ball of radius $r\le \frac{\epsilon}{2}$ centered in $M_1$ is a subset of $M$. {We shall see that $M_1$ is connected, see the claim below. Since $M$ has a finite volume, there exist positive constants $c_{\countc}(n)$ and  $c_{\countc}(n)$  such that and any two points  $x,y\in M_1$ can be connected by a sequence of $N\le c_{11}(n)r^{-n}V$ overlapping balls $\{B_j/2\}_{j=1}^N$ of radius $r/2$ centred in $M_1$, with each ball intersects at most $c_{10}(n)$ of $B_j$'s. Indeed, let $N$ be the number of pairwise disjoint balls $\{B_j/4\}$ of radius $r/4$ (each $B_j$ has radius $r$) centred in $M_1$. Then $\{B_j/2\}$ covers $M_1$ and $N\le \frac{V}{\min_j \vol_{\g}(B_j/4)}\le c_{11}(n)r^{-n}V $. The last inequality is a consequence of the G\"unther-Bishop volume comparison theorem~\cite[Theorem III.4.2]{chavel_2006} which implies that the volume of a ball of radius $r$ is bounded below by      \[\vol_{\g}(B_r)\ge \omega_n r^n.\]
     Similarly, using the Bishop-Gromov volume comparison theorem~\cite[Theorem III.4.5]{chavel_2006} for $(M,\g)$, one can show there exists $c_{10}(n)$ such that each ball intersects at most $c_{10}(n)$ balls.} 
  Therefore,
  \begin{eqnarray*}
   \sum_j  \|\nabla_{\g_\circ} \varphi\|_{\infty, B_j/2 }&\le & {c_{8}(n)}{r^{-\frac{n}{2}}}\sum_{j=1}^N\left(\int_{B_j}|\nabla_{\g_\circ}\varphi|^2\right)^{1/2} \\
   &\le& {c_{8}(n)}{r^{-\frac{n}{2}}} \sqrt{N}\left(\sum_{j=1}^N\int_{B_j}|\nabla_{\g_\circ}\varphi|^2\right)^{1/2}\\
   &\le& {c_{\countc}(n)}r^{-n}  \sqrt{V}\left(\int_{M}|\nabla_{\g_\circ}\varphi|^2\right)^{1/2}\\
   &=&{c_{12}(n)}r^{-n}  \sqrt{V}\sqrt{\sigma_1(M,\g_\circ)},
  \end{eqnarray*}
 { where $c_{12}(n)=\sqrt{c_{10}(n)c_{11}(n)}$}.
  Given the bound on $\tube_{n,\kappa}(A)$ in Lemma \ref{lem:cn} and Remark~\eqref{rem:lowerv}, we choose 
  \begin{equation}\label{varepsilon}
      r=c_{\countc}(n)A^{-\frac{1}{\kappa(n-2)}}<\epsilon/2
  \end{equation}
   for the rest of the proof.
       To summarise, we get 
  \[\sum_j  \|\nabla_{\g_\circ} \varphi\|_{\infty, B_j/2 }\le {c_{\countc}(n)}A^{\frac{n}{\kappa(n-2)}}\sqrt{V}\sqrt{\sigma_1(M,\g_{\circ})}.\]

If there exists a piece-wise geodesic  curve $c:[0,1]\to M_1$ connecting $x$ and $y$ in $M_1$, then
   \begin{eqnarray}\label{oscilation}
      \nonumber |\varphi(x)-\varphi(y)|&=&\left|\int_0^1\frac{d}{dt}\varphi\circ c(t)dt\right|\\
       \nonumber&\le& \sum_{j}\|\nabla_{\g_\circ} \varphi\|_{\infty, B_j/2 }\int_{t_j}^{t_{j+1}}|c'(t)|_{\g_\circ}dt\\
      \nonumber &\le&r\sum_{j}\|\nabla_{\g_\circ} \varphi\|_{\infty, B_j/2 }\\
       &\le& {c_{\countc}(n)}A^{\frac{n-1}{\kappa(n-2)}}\sqrt{V}\sqrt{\sigma_1(M,\g_\circ)}
   \end{eqnarray}
   Here, $0=t_0<t_1< \cdots<t_m=1$ is a partition of $[0,1]$ such that $c([t_j,t_{j+1}])\subset B_j/2$. {We want to conclude that}  the oscillation of $\varphi$ on $M_1$ is also controlled  by $\sigma_1(M,\g_{\circ})$, i.e
  {$$\displaystyle{\sup_{x,y\in M_1}|\varphi(x)-\varphi(y)|}\le {c_{15}(n)}A^{\frac{n-1}{\kappa(n-2)}}\sqrt{V}\sqrt{\sigma_1(M,\g_\circ)}.$$}
   
 For this purpose, {it is enough to show} the connectedness of $M_1$.\\

\noindent\textit{Claim.~} $M_1=M_{\thick}^\epsilon\setminus\ostube^\epsilon(\partial M)$ is connected. 
  \begin{proof}[Proof of the claim]
  {Note that $M^{\epsilon}_{\thick}$ is connected.  We can write 
  \begin{equation}\label{eq:tauepsilon}
      \ostube^\epsilon(\partial M) =\ostube^\epsilon(\partial M \cap M^{2\epsilon}_{\thick})\cup\ostube^\epsilon(\partial M \cap (\bM^{2\epsilon}\setminus \bigcup_{|\gamma|\le 4\epsilon} S_\gamma)), \end{equation}  
      where $S_\gamma\subset M^{2\epsilon}_{\thick}\cap M^{2\epsilon}_{\thin}$ is the shell of the thin tube $T_\gamma$.} We can write $M_1$ as
      $$M_1=\left(M^{\epsilon}_{\thick}\setminus\ostube^\epsilon(\partial M \cap M^{2\epsilon}_{\thick})\right)\setminus \ostube^\epsilon(\partial M \cap (\bM^{2\epsilon}\setminus \bigcup_{|\gamma|\le 4\epsilon} S_\gamma)).$$
      We first show that $M^{\epsilon}_{\thick}\setminus\ostube^\epsilon(\partial M \cap M^{2\epsilon}_{\thick})$ is connected.
      For any $p\in \partial M\cap M^{2\epsilon}_{\thick}$  and $x\in M$, we have (see \cite{Xu} and reference therein) $$\inj_{{(\double M,\g)}}(x)\ge \inj_{{(\double M,\g)}}(p)-\dist_{\g}(x,p).$$  Hence,  $\inj_{{(\double M,\g)}}(x)\ge\epsilon$ for any $x\in{\ostube^\epsilon(\partial M \cap M^{2\epsilon}_{\thick})}$ {which implies $\ostube^\epsilon(\partial M \cap M^{2\epsilon}_{\thick})\subset M^{\epsilon}_{\thick}$. Therefore, $M^{\epsilon}_{\thick}\setminus\ostube^\epsilon(\partial M \cap M^{2\epsilon}_{\thick})$ is homoemorphic (using the exponential map in the direction normal to the boundary) to $M^{\epsilon}_{\thick}$. Let $\Phi:M^{\epsilon}_{\thick} \to M^{\epsilon}_{\thick}\setminus\ostube^\epsilon(\partial M \cap M^{2\epsilon}_{\thick})$ be a homomorphism between the two. Then $\Phi$ gives also a homomorphism between $M^{\epsilon}_{\thick}\setminus U$ and $\left(M^{\epsilon}_{\thick}\setminus\ostube^\epsilon(\partial M \cap M^{2\epsilon}_{\thick})\right)\setminus U$, for any domain $U$. Hence, it is enough to show that $M^{\epsilon}_{\thick}\setminus U$ for  $U=\ostube^\epsilon(\partial M \cap (\bM^{2\epsilon}\setminus \bigcup_{|\gamma|\le 4\epsilon} S_\gamma))$ is connected.
      We claim that 
      \begin{equation}\label{eq:inclusion}
          \ostube^\epsilon(\partial M \cap (\bM^{2\epsilon}\setminus \bigcup_{|\gamma|\le 4\epsilon} S_\gamma))\subset \bM^{2\epsilon}.\end{equation}
 We then obtain that $\ostube^\epsilon(\partial M \cap (\bM^{2\epsilon}\setminus \bigcup_{|\gamma|\le 4\epsilon} S_\gamma))$ is homeomorphic to $\bM^{2\epsilon}$, and as a result $M^\epsilon_{\thick}\setminus\ostube^\epsilon(\partial M \cap (\bM^{2\epsilon}\setminus \bigcup_{|\gamma|\le 4\epsilon} S_\gamma))$ is homeomorphic to $M^\epsilon_{\thick}\setminus\bM^{2\epsilon}$.   Note that $M_{\thick}^\epsilon\setminus\bM^{2\epsilon}$ is connected {and hence so is $M^\epsilon_{\thick}\setminus\ostube^\epsilon(\partial M \cap (\bM^{2\epsilon}\setminus \bigcup_{|\gamma|\le 4\epsilon} S_\gamma))$}.
 }
 
  It remains to prove the claim \eqref{eq:inclusion}. For every $x\in \ostube^\epsilon(\partial M \cap (T_{\gamma}\setminus S_\gamma))$ with $|\gamma|\le 2\epsilon$, let $x_0\in \partial M\cap(T_{\gamma}\setminus S_\gamma)$ be a point such that $\dist_{\g}(x,x_0)=\dist_{\g}(x,\partial M)\le \epsilon$. Then 
\[\dist_{\g}(x,\gamma)\le \dist_{\g}(x,x_0)+\dist_{\g}(x_0,\gamma)\le \epsilon +\dist_{\g}(\gamma,\partial T_\gamma)-\frac{\mu}{5}\le \dist_g(\gamma,\partial T_\gamma).\]
In the last inequality, we use the fact that  $\epsilon\le\frac{\mu}{5}$, see \eqref{def:epsilon}. Therefore, $x\in T_\gamma\subset \bM^{2\epsilon}$ and it proves the inclusion \eqref{eq:inclusion}.
 \end{proof}
  Let us assume $\sigma_1(M,g_\circ)\le \frac{\delta}{b 
  A^{\frac{2n}{\kappa(n-2)}}V_1 V}$ for some $\delta$ small enough depending only on $n$ and $\kappa$, which will be determined later. We shall see that we get a contradiction. Consider the following two cases.

\noindent\textit{Case 1.~} If $\|\varphi\|_{\infty, M_1}>c_{15}{(n)} \sqrt{\frac{\delta}{bA^{\frac{2}{\kappa(n-2)}}V_1} }$, we show that $\varphi$ never vanishes on  $M_1\cup \intM^\epsilon$. By the bound on the oscillation of $\varphi$ given in \eqref{oscilation}, we have 
 \[
      |\varphi(x)-\varphi(y)|
       \le {c_{15}(n)} \sqrt{\frac{\delta}{bA^{\frac{2}{\kappa(n-2)}}V_1} },\qquad{x,y\in M_1}
  \]
which implies that $\varphi$  never vanishes on $M_1$. Moreover, if $\phi$  vanishes on $\intM^{\epsilon}\setminus\ostube^{\epsilon}(\partial M)$ which is homeomorphic to a disjoint union of type $I$ tubes, then its zero-set has to enclose a bounded domain in a type I tube and by maximum principle $\phi$ should be zero on that domain. Hence, $\phi$ does not vanish on $M_1\cup (\intM^{\epsilon}\setminus\ostube^{\epsilon}(\partial M))$. Since $M_1\cup (\intM^{\epsilon}\setminus\ostube^{\epsilon}(\partial M))$ is connected,
one of the nodal domains of $\varphi$ should be a subset of a connected component of  $\ostube^\epsilon(\partial M)\cup \bM^{\epsilon}$.
 Each connected component of $\ostube^\epsilon(\partial M)\cup \bM^{\epsilon}$ is of the form  $\ostube^\epsilon(\Sigma)\cup {M_{\thin,\Sigma}^{\epsilon}}$, where $\Sigma$ is a connected component of $\partial M$ and ${M_{\thin,\Sigma}^{\epsilon}}$ is defined in \eqref{mthinksigma}.\\
 
 Let $M_2=\ostube^\epsilon(\Sigma)\cup {M_{\thin,\Sigma}^{\epsilon}}$ be a connected component containing a nodal domain of $\varphi$. As a result of the variational characterization of Steklov eigenvalues, we have  $\sigma_1(M,\g_\circ)\ge \sigma_1^D(M_2,\g_\circ)$, where $\sigma_1^D(M_2,\g_\circ)$ is the first Steklov-Dirichlet eigenvalue of $M_2$ with Steklov condition on $\Sigma$.    By Lemma \ref{lem:stekdir},  we get
 \[\frac{\delta}{bA^{\frac{2n}{\kappa(n-2)}}V_1V}\ge\sigma_1(M,\g_\circ)\ge \sigma_1^D(M_2,\g_\circ)> c_{\countc}(n,\kappa)
 %\frac{\kappa(n-1)}{2^{n-1}}.
 \]
As discussed in Section \ref{sect: preliminaries}, $A, V_1$, and $V$ are bounded below by a constant depending only on $n$. Hence, we get a contradiction if we choose $\delta$  small enough depending only on $n$ and $\kappa$. \\
 
\noindent\textit{Case 2.~} We now consider the remaining  case, $\|\varphi\|_{\infty, M_1}\le c_{15}{(n)} \sqrt{\frac{\delta}{bA^{\frac{2}{\kappa(n-2)}}V_1}} $. 

  Since $\|\varphi\|^2_{L^2(\partial M,\g_{\circ})}=1$, there exists a connected component $M_\Sigma:=\ostube^\epsilon(\Sigma)\cup {M_{\thin,\Sigma}^{3\epsilon}}$ of $\ostube^\epsilon(\partial M)\cup \bM^{3\epsilon}$ with $\|\varphi\|^2_{L^2(\Sigma,\g_\circ)}\ge \frac{1}{b}$. Let $F=f\varphi$, where $f$ is \begin{itemize}
\item[-] equal to $1$ on $\ostube^\epsilon(\del \Sigma) \cup (M_{\thin,\Sigma}^{3\epsilon} \setminus S_\Sigma)$, where $S_\Sigma$ is the union of all shells in ${M_{\thin,\Sigma}^{{3}\epsilon}}$,
\item[-]identically zero on $M\setminus (M_\Sigma\cup\ostube^{2\epsilon}(\Sigma))$,
\item[-]and Lipschitz on $S_\Sigma\cup (\ostube^{2\epsilon}(\Sigma)\setminus M_\Sigma)$  with Lipschitz constant $\Lip_{\g_{\circ}}f=L\le\frac{c_{\countc}}{\epsilon}$, where $c_{17}$ is a universal constant. 
\end{itemize}
An example of such $f$ is the following Lipschitz function:
$$f(x) = \max\left(1 - \frac{\dist_{\g} (x, \ostube^\epsilon(\del \Sigma) \cup (M_{\thin,\Sigma}^{3\epsilon} \setminus S_\Sigma)}{\epsilon}, 0\right). $$
Note that from the definition of the shell given in Section \ref{sect: preliminaries}, we have $\dist_{\g}({M_{\thin,\Sigma}^{3\epsilon}}\setminus S_{\Sigma},M\setminus M_{\thin,\Sigma}^{3\epsilon})= \frac{\mu}{5}.$ Since $\epsilon < \frac{\mu}{5}$, one then see that the support of $f$ is contained in $M \setminus (M_\Sigma\cup\ostube^{2\epsilon}(\Sigma))$. Furthermore,
taking into account the relation between $\g$ and $\g_\circ$, we have $$\Lip_{\g_\circ} f\le\sqrt{2}\Lip_{\g} f\le\frac{\sqrt{2}}{\epsilon}.$$

Let $\tilde M_\Sigma={M_{\thin,\Sigma}^{3\epsilon}}\cup\ostube^{2\epsilon}(\Sigma)$ and consider the Steklov-Dirichlet problem on $\tilde M_\Sigma$
with Steklov condition on $\Sigma$. The function $F$ can be used as a test-function for this problem and
$$\sigma_1^D(\tilde M_\Sigma,\g_\circ)\le \frac{\int_{\tilde M_\Sigma}|\nabla_{\g_{\circ}} F|^2}{\int_\Sigma F^2}\le \frac{\int_{\tilde M_\Sigma}|\nabla_{\g_\circ} {F}|^2}{\int_\Sigma {\varphi}^2}=b\int_{\tilde M_\Sigma}|\nabla_{\g_\circ} F|^2.$$
We now obtain an upper bound for ${\int_{\tilde M_\Sigma}|\nabla F|^2}$ in terms of $\sigma_1(M,\g_\circ)$.
\allowdisplaybreaks
\begin{align}
\nonumber    {\int_{\tilde M_\Sigma}|\nabla_{\g_\circ} F|^2}&={\int_{\tilde M_\Sigma}|f\nabla_{\g_\circ} \varphi+\varphi\nabla_{\g_\circ} f|^2}\\
  \nonumber  &\le  \left(\left(\int_{\tilde M_\Sigma}|\nabla_{\g_\circ}\varphi|^2\right)^{\frac{1}{2}}+L\left(\int_{{\supp(\nabla f)}}\varphi^2\right)^{\frac{1}{2}}\right)^2
\\
\label{eq:26}&\le  \left(\sqrt{\sigma_1(M,\g_\circ)}+c_{\countc}(n)A^{\frac{1}{\kappa(n-2)}}\,\|\varphi\|_{\infty,\supp(\nabla f)}\vol_{\g}(\tilde M_\Sigma)^{\frac{1}{2}}\right)^2\\
\label{eq:27}&{\le  \left(\sqrt{\sigma_1(M,\g_\circ)}+c_{18}(n)A^{\frac{1}{\kappa(n-2)}}\,\|\varphi\|_{\infty,M_1}{\sqrt{V_1}}\right)^2}\\
\nonumber&\le  \left(\sqrt{\sigma_1(M,\g_\circ)}+c_{\countc}(n)\, \sqrt{\frac{\delta }{b\, }}\right)^2\\
\nonumber&=\frac{\delta}{b} \left(\frac{1}{A^{\frac{n}{\kappa(n-2)}}\sqrt{V_1V}}+\,c_{19}(n) \right)^2.
\end{align}
We use $\vol_{\g}$ instead of $\vol_{\g_\circ}$ in inequality \eqref{eq:26} since the two metrics are quasi-isometric, with the constant multiple accounted for in $c_{18}(n)$.  We also replaced $L$ by its upper bound, recall the definition of $\epsilon$ in \eqref{def:epsilon}. In \eqref{eq:27}, we use the fact that $$\supp(\nabla f)=S_{\Sigma}\cup\left(\ostube^{2\epsilon}(\Sigma)\setminus\left(\ostube^{\epsilon}(\Sigma)\cup M^{3\epsilon}_{\thin,\Sigma}\right)\right)\subset M_{\thick}^\epsilon.$$ Indeed, by definition, 
$$S_\Sigma\subset M_{\thick}^{3\epsilon}\subset M_{\thick}^{\epsilon},\qquad\ostube^{2\epsilon}(\Sigma)\setminus\left(\ostube^{\epsilon}(\Sigma)\cup M^{3\epsilon}_{\thin,\Sigma}\right)\subset \ostube^{2\epsilon}(\Sigma\cap M_{\thick,\Sigma}^{3\epsilon}).$$
For every $x,p\in M$, we have (see \cite{Xu} and reference therein) $$\inj_{{(\double M,\g)}}(x)\ge \inj_{{(\double M,\g)}}(p)-\dist_{\g}(x,p).$$ Taking $x\in \ostube^{2\epsilon}(\Sigma\cap M_{\thick,\Sigma}^{3\epsilon})$ and $p\in \Sigma \cap M_{\thick,\Sigma}^{3\epsilon} $ with $\dist(x,p)=2\epsilon$,  we obtain that $\inj_{(\double M,\g)}(x)\ge \epsilon$. Hence $x\in M_{\thick}^\epsilon$.
Therefore, by Lemma \ref{lem:stekdir} and Remark \ref{rem:sobolevinq}, we get 
\[c_{\countc}(n,\kappa)<\sigma_1^D(\tilde M_\Sigma,\g_\circ)\le {\delta} \left(\frac{1}{A^{\frac{n}{\kappa(n-2)}}\sqrt{V_1V}}+\,c_{19}(n) \right)^2.\]
Recall that $A$, $V_1$ and $V$ have a uniform lower bound depending only on $n$. Hence, we can choose $\delta$ (depending only on $n$ and $\kappa$) small enough to  get a contradiction. It completes the proof. 
\end{proof}

\begin{remark}\label{rem:HMP}
   In \cite{HMP}, the argument presented above is  {adapted} to the case of surfaces leading to an improvement of Perrin's result \cite{Per22} for hyperbolic surfaces. In dimension two,  the boundary is a collection of simple closed geodesics. Thus,  some steps of the above argument can be simplified.  However, the main technical challenge is that the thick part is not necessarily connected. 
\end{remark}

\section{Constructing small Steklov eigenvalues}\label{sec:smallstek}
\setcounter{countc}{0}
In this section, we give examples showing the necessity of the presence of the volume of the manifold and {the volume of} its boundary in the lower bounds given in  Theorems \ref{thm:Schoencounterpart} and \ref{thm:specgap-intro}. \\
The first example not only shows that the presence of $V$ is necessary when  $b>1$ but also the exponent of  $V$ in Theorem \ref{thm:specgap} is sharp. 
\begin{example}\label{ex:smalleigenvalue} Given any $(-\kappa^2)$–\,pinched $n$–\,manifold $M$ with $b>1$ totally geodesic boundary components, one can construct a family $\{M_j\}_{j=1}^\infty$ of $(-\kappa^2)$–\,pinched $n$–\,manifold by attaching copies of $M$ such that  each $M_j$ has $b$ totally geodesic boundary components,  
the volume of $\partial M_j$ is uniformly bounded,
    and $$\lim_{j\to\infty}\sigma_{b-1}(M_j)=0.$$ 
Moreover, the rate of decay is of order $\frac{1}{\vol(M_j)}$.

Let $M_0$ be defined as $M$ with $b$ boundary components $\{\Sigma_1,\cdots,\Sigma_b\}$, and let $M_1^{(i)}$,  $1\le i\le b$, be obtained by attaching two copies of $M_0$ along their identical boundaries, except for $\Sigma_i$. As a result, $M_1^{(i)}$ is a manifold with two boundary components, each isometric to $\Sigma_i$.
Inductively, let $M_j^{(i)}$  be a manifold obtained by attaching $M_1^{(i)}$ to $M_{j-1}^{(i)}$ along one of its boundary components. Note that $M_j^{(i)}$ also has two boundary components, each isometric to $\Sigma_i$. 
%We  now define $M_{j}$,  $j\geq 1$, as the manifold obtained by attaching $M_1^{(i)}$ to $M_{j-1}$ along $\Sigma_i$, $1\le i\le b$. 
We  now define $M_{j}$,  $j\geq 1$, as the manifold obtained by attaching $M_{j}^{(i)}$ to $M_0$ along $\Sigma_i$, $1\le i\le b$. The boundary of each $M_j$ is isometric to $\partial M=\sqcup_{i}\Sigma_i$. See Figure \ref{Pic:3} for an illustration of this construction. 

\begin{figure}[h]
    \centering
    \def\svgwidth{0.5\textwidth}
    %% Creator: Inkscape 1.1.2 (0a00cf5339, 2022-02-04), www.inkscape.org
%% Accompanies image file 'small_eig.pdf' (pdf, eps, ps)
\begingroup%
  \makeatletter%
  \providecommand\color[2][]{%
    \errmessage{(Inkscape) Color is used for the text in Inkscape, but the package 'color.sty' is not loaded}%
    \renewcommand\color[2][]{}%
  }%
  \providecommand\transparent[1]{%
    \errmessage{(Inkscape) Transparency is used (non-zero) for the text in Inkscape, but the package 'transparent.sty' is not loaded}%
    \renewcommand\transparent[1]{}%
  }%
  \providecommand\rotatebox[2]{#2}%
  \newcommand*\fsize{\dimexpr\f@size pt\relax}%
  \newcommand*\lineheight[1]{\fontsize{\fsize}{#1\fsize}\selectfont}%
  \ifx\svgwidth\undefined%
    \setlength{\unitlength}{723.43103476bp}%
    \ifx\svgscale\undefined%
      \relax%
    \else%
      \setlength{\unitlength}{\unitlength * \real{\svgscale}}%
    \fi%
  \else%
    \setlength{\unitlength}{\svgwidth}%
  \fi%
  \global\let\svgwidth\undefined%
  \global\let\svgscale\undefined%
  \makeatother%
  \begin{picture}(1,0.89554906)%
    \lineheight{1}%
    \setlength\tabcolsep{0pt}%
    \put(0,0){\includegraphics[width=\unitlength,page=1]{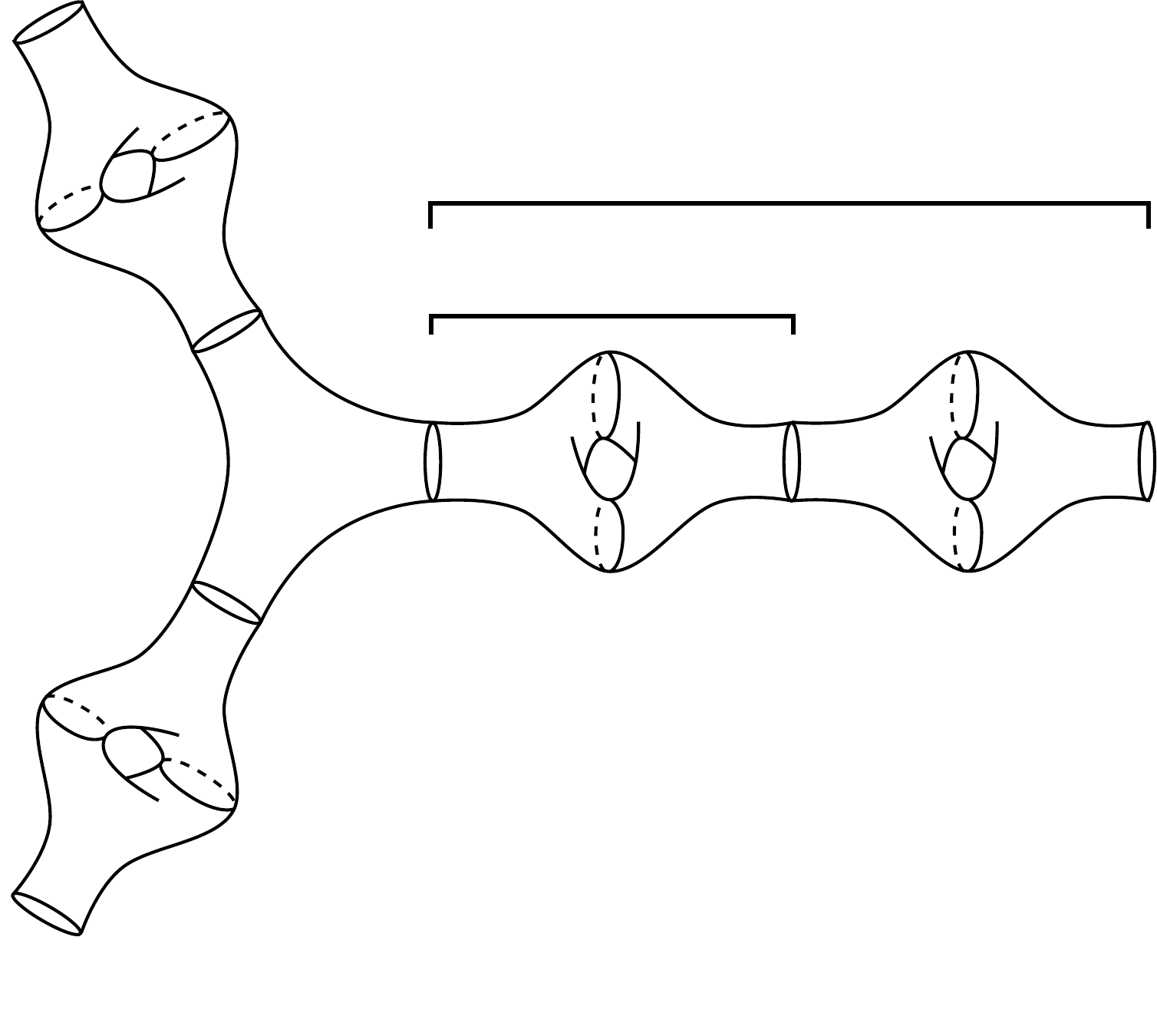}}%
    \put(0.47786216,0.64758819){\color[rgb]{0,0,0}\makebox(0,0)[lt]{\lineheight{1.25}\smash{\begin{tabular}[t]{l}$M_1^{(i)}$\end{tabular}}}}%
    \put(0.63687278,0.74763919){\color[rgb]{0,0,0}\makebox(0,0)[lt]{\lineheight{1.25}\smash{\begin{tabular}[t]{l}$M_2^{(i)}$\end{tabular}}}}%
    \put(0,0){\includegraphics[width=\unitlength,page=2]{small_eig.pdf}}%
    \put(0.31814324,0.00672755){\color[rgb]{0,0,0}\makebox(0,0)[lt]{\lineheight{1.25}\smash{\begin{tabular}[t]{l}$M_1$\end{tabular}}}}%
    \put(0.23079409,0.47847537){\color[rgb]{0,0,0}\makebox(0,0)[lt]{\lineheight{1.25}\smash{\begin{tabular}[t]{l}$M_0$\\\\\end{tabular}}}}%
  \end{picture}%
\endgroup%

    \caption{This picture illustrates the construction of the sequence $M_j$ when $M=M_0$ is an $n$–\,manifold with three totally geodesic boundary components.}
    \label{Pic:3}
\end{figure}

We now construct $b$ test functions $\{f_{j,1},\cdots,f_{j,b}\}$ on $M_{j}$ with mutually disjoint support. 
\[f_{j,i}(x):=\begin{cases} 1-\frac{\dist(x,\Sigma_i)}{\dist(M_j\setminus M_{j}^{(i)},\,\Sigma_i)},&x\in M_j^{(i)} \\
 0&x\in M_j\setminus M_j^{(i)},
\end{cases}\]
 Using the variational  characterisation of $\sigma_{b-1}(M_j)$, 
 \[\sigma_{b-1}(M)=\inf_{E}\sup_{0\ne f\in E}\frac{\int_{M}|\nabla f|^2\dd v}{ \int_{\partial M}f^2\dd s}\]
where $E\subset{H^1(M)}$ and its trace onto $L^2(\partial M)$ gives a $b$–\,dimensional subspace of $L^2(\partial M)$, together with the inequality 
\[\dist(M_j\setminus M_{j}^{(i)},\,\Sigma_i)\ge 2j \tube_{n,\kappa}(A),\]
we get
\begin{eqnarray*}
    \sigma_{b-1}(M_j)&\le& \max_{1\le i\le b}\frac{\int_{M_j}|\nabla f_{j,i}|^2\dd v}{\int_{\partial M_j}f_{j,i}^2\dd s}\\&\le &\frac{V}{j\tube_{n,\kappa}^2(A)A}\\&\le& \frac{(b+1)V^2}{\tube_{n,\kappa}^2(A)A}\frac{1}{\vol(M_j)}\underset{j\to\infty}{\longrightarrow} 0,
\end{eqnarray*}
where $V=\vol(M)$ and $\displaystyle{A=\max_{1\le i\le b}\vol(\Sigma_i)}$.\hfill\qed
\end{example}
We now show that the presence of the {volume} of the boundary in the lower bound given in Theorem \ref{thm:Schoencounterpart} is necessary. 
\begin{theorem}\label{thm:smalleigenvalues}
   For any  given $k \in \mathbb{N}$ and any $0<\varepsilon<\eta$, where {$\eta$} is the  constant {defined in Section \ref{sect: preliminaries}}, there exists a  hyperbolic $n$–\,manifold $M_\varepsilon$, $n\ge3$, with two totally geodesic boundaries such that the distance between two boundary components is less than $\epsilon$ and
    \begin{align*}
        \sigma_k(M_\varepsilon) \leq c_{\countc}(k,n) \epsilon. 
    \end{align*} 
    In particular, the volume of $\partial M_\varepsilon$ tends to infinity as $\varepsilon\to0$.
\end{theorem} 

\begin{remark}\label{rmk:small-eigenvalues}
It is not necessary to have at least two boundary components in order to construct small Steklov eigenvalues. The above proof can be applied to any sequence of manifolds with a connected totally geodesic boundary, provided that there exists a geodesic arc with endpoints meeting the boundary orthogonally, and its length going to zero. 

For example 
    in dimension 3, using the model for random hyperbolic 3-manifolds constructed in \cite{PR22}, there exist hyperbolic 3-manifolds with a single connected totally geodesic boundary component with arbitrarily small arcs which are homotopically non-trivial with respect to the boundary. The previous construction can then be applied to these manifolds. 
\end{remark}
\begin{proof}[Proof of Theorem \ref{thm:smalleigenvalues}]

Following the construction in \cite[Section 2]{BT11},  for any $\varepsilon > 0$, there exists a hyperbolic $n$–\,manifold $M_\varepsilon$ with two totally geodesic boundary components $\Sigma^\varepsilon_1$ and $\Sigma^\varepsilon_2$ such that the length of the minimal arc $\alpha$ homotopically non-trivial relative to $\del M_\varepsilon$ is less than $\varepsilon$. {The arc} $\alpha$ is a geodesic orthogonal to the boundary $\del M_\varepsilon$.  The tubular neighborhood theorem in the hyperbolic setting  
implies that the boundary volume should go to infinity as $\varepsilon$ tends to zero.

Without loss of generality, we assume the length $\length(\alpha) = \varepsilon$. 
Let  $p_1\in \Sigma_1^\varepsilon,$ and $ p_2\in \Sigma_2^\varepsilon$ be the endpoints of $\alpha$. 
Since $\varepsilon<{\eta}$, {the arc} $ \alpha$ is in the thin part of $ M_\varepsilon$. Hence,  we have a tube $T_\alpha\subset {({M}_{\varepsilon})_{\thin,\partial}}$  diffeomorphic to  $ \alpha \times B^{n-1}$ where $B^{n-1}$ is $(n-1)$–\,dimensional ball in $\R^{n-1}$. {Let $R_\varepsilon$ denote the distance between $\double\alpha$ and the boundary of $\double T_\alpha\subset \double M_\epsilon$. It} is strictly bigger than 1 (due to the choice of {$\eta$} in Section \ref{sect: preliminaries}), and moreover, $R_\varepsilon$ goes to $\infty$ when $\varepsilon$ tends to 0.  \\
We parameterise $T_\alpha$ using Fermi coordinates 
$(t, r, \phi)$ around $\alpha$, where $t\in (0,\varepsilon)$, $r\in(0,R_\varepsilon)$, and $ \phi\in\mathbb{S}^{n-2}$. The hyperbolic metric $h$ on $T_\alpha$ can be expressed as (see e.g. \cite{Bus80})
\[h=\cosh^2r\dd t^2+\dd r^2+\sinh^2r \dd\phi^2,\]
where $\dd\phi^2$ is the metric on the Euclidean sphere $\mathbb{S}^{n-2}$.
For every  $t\in(0,\varepsilon)$ and $R\in (0,R_\varepsilon)$ we have   embedded   balls ${B}_t=\{(t,r,\phi)\in T_\alpha: r\in(0,R), \phi\in \mathbb{S}^{n-2}\}$, and isometric to the hyperbolic ball $\hball(R)\subset\mathbb{H}^{n-1}$ of radius $R$.    \\
 Let $R=1$ and $\lambda^D_k(\hball(1))$ be the Dirichlet eigenvalues of the unit hyperbolic ball $\hball(1)$ and $\{u_k\}$ be the corresponding  eigenfunctions normalized by $\|u_k\|_{L^2(\hball(1))} = 1$. 
We define the function $U_k: M_\varepsilon\to \mathbb{R}$ as follows. For $x\in M_\varepsilon\setminus T_{\alpha}$, we set $U_k(x)$ to be equal to 0.  For $x\in T_\alpha$, the function $U_k$ expressed in Fermi coordinates takes the form of

\begin{align*}
    U_k(t,r,\phi) =\begin{cases}
      u_k(r,\phi)& (r,\phi)\in(0,1)\times \mathbb{S}^{n-1} \\
      0& \text{otherwise}
    \end{cases}
\end{align*}

Since by construction  $\partial_t U_k = 0$, we have 
$|\nabla U_k|^2 = |\nabla^{B_t} U_k|^2$ on $T_\alpha$.
Thus,
\begin{align*}
    \int_{T_\alpha} |\nabla U_k|^2 \dd v &= \varepsilon\int_0^1\int_{\mathbb{S}^{n-2}} |\nabla^{B_t} U_k|^2 \sinh^{n-2}r \cosh r \dd r \, d\phi\\& 
\le \varepsilon\cosh(1)\int_0^1\int_{\mathbb{S}^{n-2}} |\nabla^{B_t} U_k|^2 \sinh^{n-2}r \dd r \, d\phi \\&=  \varepsilon\cosh(1)\lambda_k^D(\hball(1)).
\end{align*}
The test functions  $\{U_j\}_{j=1}^{k+1}$  span a $(k+1)$–\,dimensional subspace of $L^2(\partial M_\varepsilon)$. Therefore, by the variational characterisation of the Steklov eigenvalues, {we conclude that}
\begin{align*}
    \sigma_k(M_\varepsilon) &\leq \max_{j=1,\dots,k+1} \frac{\int_{T_\alpha} |\nabla U_j|^2}{\int_{\del T_\alpha} U_j^2} \leq \frac{\cosh(1)}{2} \lambda_{k+1}^D(\hball(1))\varepsilon.
\end{align*}
\end{proof}
In the proof above, one has the flexibility to substitute the radius of the ball with any value within the range of $(0, R_\varepsilon)$. However, such a substitution does not give an improved rate of decay, because of the following limiting behavior due to Berge \cite{Berge2023}, $$\lim_{R\to\infty}\lambda_{k+1}^D(\hball(R))=\frac{(n-1)^2}{4}.$$

\bibliographystyle{alpha}
\bibliography{HyperbolicRef}
\end{document}